\documentclass[11pt,twoside,reqno]{amsart}
\usepackage[width=13cm, height=22cm]{geometry}
\usepackage{microtype}
\usepackage[OT1]{fontenc}
\usepackage{amssymb}
\usepackage{mathrsfs}
\usepackage{enumerate}

\usepackage{amsmath, amsthm}
\usepackage{color}

\numberwithin{equation}{section}

\newtheorem{theorem}{Theorem}[section]
\newtheorem{corollary}[theorem]{Corollary}
\newtheorem{conjecture}[theorem]{Conjecture}
\newtheorem{lemma}[theorem]{Lemma}
\newtheorem{proposition}[theorem]{Proposition}
\newtheorem{definition}[theorem]{Definition}
\newtheorem{remark}[theorem]{Remark}
\newtheorem*{remark*}{Remark}

\newtheorem*{claim*}{Claim}
\newtheorem{problem}[theorem]{Problem}

\newtheorem*{assertion*}{Mahler's Assertion}

\DeclareMathOperator{\Haus}{\mathcal{H}}

\DeclareMathOperator{\Prob}{\mathbb{P}}

\DeclareMathOperator{\diam}{diam}
\DeclareMathOperator{\dimh}{dim_H}

\allowdisplaybreaks

\renewcommand{\epsilon}{\varepsilon}

\renewcommand{\Bbb}[1]{\mathbb{#1}}
\newcommand{\N}{{\Bbb N}}         % natural numbers
\newcommand{\I}{{\Bbb I}}
\newcommand{\R}{{\Bbb R}}        % real numbers
\newcommand{\Z}{{\Bbb Z}}         % integer numbers
\newcommand{\T}{{\Bbb T}}	%Torus

\newcommand{\cA}{{\mathcal A}}
\newcommand{\cB}{{\mathcal B}}

\newcommand{\cH}{{\mathcal H}}

\newcommand{\cL}{{\mathcal L}}
\newcommand{\cM}{{\mathcal M}}

\newcommand{\cR}{{\mathcal R}}

\newcommand{\cW}{{\mathcal W}}

\newcommand{\Om}{\Omega}
\newcommand{\U}{\Upsilon}
\newcommand{\La}{\Lambda}

\newcommand{\ie}{{\it i.e.}\/ }

\newcommand{\dist}{\operatorname{dist}}

\newcommand{\vv}[1]{{\mathbf{#1}}}

\newcommand{\x}{\mathbf{x}}
\newcommand{\y}{\mathbf{y}}
\newcommand{\p}{\mathbf{p}}
\newcommand{\q}{\mathbf{q}}

\newcommand{\veca}{\mathbf{a}}
\newcommand{\vecb}{\mathbf{b}}
\newcommand{\0}{\mathbf{0}}

\begin{document}

\title{The Mass Transference Principle: Ten Years On}

\author{Demi Allen}
\address{Demi Allen\\
	Department of Mathematics \\
         University of York \\
         YO10 5DD \\
         UK}
\email{dda505@york.ac.uk}

\author{Sascha Troscheit}
\address{Sascha Troscheit \\
	Department of Pure Mathematics \\		
		University of Waterloo \\
		200 University Avenue West \\
		Waterloo, ON \\
		N2L 3G1 \\
		Canada}
\email{stroscheit@uwaterloo.ca}

%\thanks{}
\date{\today}

\begin{abstract}
  In this article we discuss the Mass Transference Principle due to Beresnevich and Velani and survey several generalisations and variants, both deterministic and random. Using a Hausdorff measure analogue of the inhomogeneous Khintchine--Groshev Theorem, proved recently via an extension of the Mass Transference Principle to systems of linear forms, we give an alternative proof of a general inhomogeneous Jarn\'{\i}k--Besicovitch Theorem which was originally proved by Levesley. We additionally show that without monotonicity Levesley's theorem no longer holds in general. 
  Thereafter, we discuss recent advances by Wang, Wu and Xu towards mass transference principles where one transitions from $\limsup$ sets defined by balls to $\limsup$ sets defined by rectangles (rather than from ``balls to balls'' as is the case in the original Mass Transference Principle). Furthermore, we consider mass transference principles for transitioning from rectangles to rectangles and extend known results using a slicing technique.
  We end this article with a brief survey of random analogues of the Mass Transference Principle.
\end{abstract}

\maketitle

\section{Introduction}
Since its discovery by Beresnevich and Velani in 2006, the Mass Transference Principle has become an important tool in metric number theory.
Originally motivated by the desire for a Hausdorff measure version of the Duffin--Schaeffer conjecture, the Mass Transference Principle allows us to transfer a Lebesgue measure statement for a $\limsup$ set defined by a sequence of balls in $\R^k$ to a Hausdorff measure statement for a related $\limsup$ set. Over the past few years a number of generalisations have been proved and more general settings have been considered.
In this article we survey several of these recent developments and consider some of their applications, mostly in the field of metric number theory.

\subsection{Notation and Basic Definitions}

Throughout, by a \emph{dimension function} we mean a continuous, non-decreasing function $f \, : \, \R^+ \to \R^+ $ such that $f(r)\to 0$ as $r\to 0 \, $. Recall that $\R^+ = [0,\infty)$. If there exists a constant $\lambda > 1$ such that for $x > 0$ we have $f(2x) \leq \lambda f(x)$ then we say that $f$ is \emph{doubling}. 

\begin{definition}
Let $F\subseteq\R^{k}$ and let $\delta>0$. The \emph{$\delta$-Hausdorff pre-measure} of $F$ with respect to the dimension function $f$, denoted $\Haus_{\delta}^{f}(F)$, is given by
\[
\Haus_{\delta}^{f}(F)=\inf\left\{\sum_{i=1}^{\infty}f(\diam(U_{i})) \;:\; F\subseteq \bigcup_{i=1}^{\infty} U_{i} \text{ and }\diam(U_{i}) \leq\delta \text{ for all }i\right\},
\]
 where the infimum is taken over all countable collections $\{U_{i}\}$ of open sets. The \emph{Hausdorff content} $\Haus_{\infty}^{f}$ with respect to $f$ is
\[
\Haus_{\infty}^{f}(F)=\inf_{\delta>0}\Haus_{\delta}^{f}(F).
\]
The \emph{Hausdorff measure} $\Haus^{f}$ with respect to $f$ is defined by
\[
\Haus^{f}(F)=\lim_{\delta\to0}\Haus^{f}_{\delta}(F).
\]
\end{definition}
We note that for all dimension functions $f$, and all bounded subsets $F\subset\R^{k}$, the Hausdorff content satisfies $\Haus_{\infty}^{f}(F)\leq f(\diam(F))$ and for all $\delta>0$
\[
\Haus_{\infty}^{f}(F)\leq\Haus_{\delta}^{f}(F)<\infty.
\]

We also observe that, for a given $f$, the $\delta$-Hausdorff pre-measure $\Haus_{\delta}^{f}(F)$ is non-decreasing as $\delta\to0$. So, using monotone convergence, the limit $\lim_{\delta\to0}\Haus_{\delta}^{f}(F)$ exists but may be infinite. 

Often we are interested in Hausdorff dimension and the classical Hausdorff $s$-measure. The Hausdorff $s$-measure, which we will usually denote by $\cH^s$, can be obtained by letting $f(r)=r^{s}$. The Hausdorff dimension of a set $F$ is then defined as follows.

\begin{definition}
Let $F\subseteq\R^{k}$. The \emph{Hausdorff dimension} of $F$ is
\[
\dim_{H}F=\inf\big\{s>0 \; : \; \Haus^{s}(F)=0\big\}.
\] 
\end{definition}

One interesting property of the Hausdorff measure is that for subsets of $\R^k$, $\cH^k$ is comparable to the $k$-dimensional Lebesgue measure. For a set $X \subset \R^k$ we denote the $k$-dimensional Lebesgue measure by $|X|$. 
Lebesgue null sets, \ie sets $X$ with $\lvert X\rvert=0$, can still have intricate geometric structure and in many cases we are able to appeal to Hausdorff dimension to discriminate between their respective `sizes'.
For further information regarding Hausdorff measures and dimension we refer the reader to \cite{Falconer ref,Mattila ref, Rogers book}.
Finally, we recall the notion of a $\limsup$ set.
\begin{definition}
Let $(A_{i})_{i \in \N}$ be a collection of subsets of a set $Y$. Then 
\[
\limsup_{i}A_{i}=\bigcap_{k=1}^{\infty}\bigcup_{i=k}^{\infty}A_{i}.
\]
Equivalently,
\[
\limsup_{i}A_{i}=\{x\in Y\,:\, x\in A_{i}\text{ for infinitely many }i\in\N\}.
\]
\end{definition}

\section{The Mass Transference Principle}
The main object of study in this article is the \emph{Mass Transference Principle} and its generalisations and variants. The original Mass Transference Principle was developed by Beresnevich and Velani in \cite{BV MTP} and was motivated by a conjecture of Duffin and Schaeffer. 

Given an approximating function $\psi: \N \to \R^+$, for $k \in \N$, let 
\[\cA_k(\psi):= \left\{\x \in \I^k : \max_{1 \leq i \leq k}\left|x_i-\frac{p_i}{q}\right| < \frac{\psi(q)}{q} \text{ for infinitely many } (\p,q) \in \Z^k \times \N\right\},\]
where $\mathbf{p}=(p_{1},p_{2},\dots,p_{k})$, be the \emph{simultaneously $\psi$-approximable points} in the unit cube, $\I^k = [0,1]^k$, and consider the following classical theorem by Khintchine~\cite{kh}. 

\begin{theorem}[Khintchine~\cite{kh}] \label{Khintchine's Theorem}
	Let $\psi: \N \to \R^+$ be an approximating function. Then
	\[
	 \lvert \cA_1(\psi)\rvert=\begin{cases}
	 0 &\text{if }\sum_{q=1}^{\infty}  {\psi(q)}<\infty,\\&\\
	 1& \text{if }\sum_{q=1}^{\infty}  {\psi(q)}=\infty\text{ and $\psi$ is monotonic.}
	 \end{cases}
	 \]
\end{theorem}

Khintchine also extended this result to the simultaneously $\psi$-approximable points in higher dimensions.

\begin{theorem}[Khintchine~\cite{Khintchine ref}] \label{Khintchine's Simultaneous Theorem}
Let $\psi: \N \to \R^+$ be an approximating function. Then
	\[
	 \lvert \cA_k(\psi)\rvert=\begin{cases}
	 0 &\text{if }\sum_{q=1}^{\infty}  {\psi(q)^k}<\infty,\\&\\
	 1& \text{if }\sum_{q=1}^{\infty}  {\psi(q)^k}=\infty\text{ and $\psi$ is monotonic.}
	 \end{cases}
	 \]
\end{theorem}

In the one-dimensional case Duffin and Schaeffer~\cite{Duffin-Schaeffer ref} constructed a counter-example showing that the full Lebesgue measure statement can fail for non-monotonic $\psi$. They also posed a conjecture on what should be true when considering general (not necessarily monotonic) approximating functions.

Given an approximating function $\psi: \N \to \R^+$ and an integer $k \geq 1$ let us denote by $\cA_k'(\psi)$ the set of points $\x \in \I^k$ such that 
\begin{align} \label{simultaneous inequality}
\left|\x - \frac{\p}{q}\right| < \frac{\psi(q)}{q}
\end{align}
for infinitely many $(\p,q) \in \Z^k \times \N$ with $\gcd(\p,q) := \gcd(p_1,\dots,p_k,q) = 1$.

\begin{conjecture}[Duffin--Schaeffer Conjecture~\cite{Duffin-Schaeffer ref}]\label{conj:duffinschaeffer1}
Let $\psi: \N \to \R^+$ be any approximating function and denote by $\phi(q)$ the Euler function. If
\[\sum_{q=1}^{\infty}\phi(q)\frac{\psi(q)}{q}=\infty \quad \text{then} \quad \lvert \cA_1'(\psi)\rvert=1.\]
\end{conjecture}

For $k \geq 2$ the analogous conjecture was formulated by Sprind\v{z}uk \cite[Chapter 1, Section 8]{Sprindzuk ref}. The conjecture depends again on slightly different coprimality conditions. Therefore, for any approximating function $\psi: \N \to \R^+$, let us denote by $\cA_k''(\psi)$ the set of points $\x \in \I^k$ for which the inequality (\ref{simultaneous inequality}) is satisfied for infinitely many $(\p,q) \in \Z^k \times \N$ which also have $\gcd(p_i,q) = 1$ for all $1 \leq i \leq k$.

\begin{conjecture}[Higher-Dimensional Duffin--Schaeffer Conjecture~\cite{Sprindzuk ref}] \label{higher dimensional Duffin-Schaeffer conjecture}
Let $\psi: \N \to \R^+$ be any approximating function and denote by $\phi(q)$ the Euler function. If
\[ \sum_{q=1}^{\infty}\phi(q)^{k}\frac{\psi(q)^{k}}{q^{k}}=\infty \quad \text{then} \quad \lvert \cA_k''(\psi)\rvert=1.\]
\end{conjecture}

For $k>1$ Sprind\v{z}uk's conjecture (Conjecture \ref{higher dimensional Duffin-Schaeffer conjecture}) was proved in the affirmative by Pollington and Vaughan~\cite{Pollington90}.

Finding a general Hausdorff measure analogue of the Duffin--Schaeffer conjecture inspired the Mass Transference Principle that we will now state.
Let $f$ be a dimension function and $\cH^f(\,\cdot\,)$ denote Hausdorff $f$-measure. Given a ball $B = B(x,r)$ in $\R^k$ of radius $r$ centred at $x$ let $B^f = B(x,f(r)^{\frac{1}{k}})$. We write $B^{s}$ instead of $B^{f}$ if $f(x) = x^s$ for some $s > 0$. In particular, we have $B^k = B$. 

\begin{theorem}[{\bf Mass Transference Principle}, Beresnevich -- Velani \cite{BV MTP}]\label{mtp theorem}
	Let $\{B_i=B(x_{i},r_{i})\}_{i\in\N}$ be a sequence of balls in $\R^k$ with
	$r_{i}\to 0$ as $i\to\infty$. Let $f$ be a dimension function
	such that $x^{-k}f(x)$ is monotonic and let $\Omega$ be a ball in $\R^k$. Suppose that, for any ball $B$ in $\Omega$,
	\[ \cH^k\big(\/B\cap\limsup_{i\to\infty}B^f_i{}\,\big)=\cH^k(B) \ .\]
	Then, for any ball $B$ in $\Omega$,
	\[ \cH^f\big(\/B\cap\limsup_{i\to\infty}B^k_i\,\big)=\cH^f(B) \ .\]
\end{theorem}

\begin{remark*}
Strictly speaking, the statement of the Mass Transference Principle given initially by Beresnevich and Velani, \cite[Theorem 2]{BV MTP}, corresponds to the case where $\Omega$ is taken to be $\R^k$ in Theorem \ref{mtp theorem}. The statement we have opted to give above is a consequence of \cite[Theorem 2]{BV MTP}.
\end{remark*}

The Mass Transference Principle allows us therefore to transfer a Lebesgue measure statement for a $\limsup$ set of balls to a Hausdorff measure statement for a $\limsup$ set of balls which are obtained by ``shrinking'' the original balls in a certain manner according to $f$. This is a remarkable result given that Lebesgue measure can be considered to be much `coarser' than Hausdorff measure. 

The Mass Transference Principle was used to show that the Duffin--Schaeffer conjecture for Lebesgue measure gives rise to an analogous statement for Hausdorff measures.

\begin{conjecture}[Hausdorff Measure Duffin--Schaeffer Conjecture \cite{BV MTP}]\label{conj:duffinschaeffer2}
Let $\psi: \N \to \R^+$ be any approximating function and let $f$ be a dimension function such that $r^{-k}f(r)$ is monotonic. If
\[ \sum_{q=1}^{\infty}{\phi(q)^{k}f\left(\frac{\psi(q)}{q}\right)}=\infty \quad \text{then} \quad \Haus^{f}(\cA_k''(\psi))=\Haus^{f}(\I^{k}).\]
\end{conjecture}

Setting $f(r)=r^k$ in the above we see that we immediately recover Conjecture \ref{higher dimensional Duffin-Schaeffer conjecture}. What is much more surprising is that, using the Mass Transference Principle (Theorem~\ref{mtp theorem}), Beresnevich and Velani proved that Conjecture~\ref{higher dimensional Duffin-Schaeffer conjecture} implies Conjecture~\ref{conj:duffinschaeffer2} and hence that they are equivalent. In particular, Conjecture~\ref{conj:duffinschaeffer2} holds for $k>1$ for general approximating functions $\psi$ and for $k=1$ if $\psi$ is furthermore monotonic, see~\cite{BV MTP}.

Two further easy yet surprising consequences of the Mass Transference Principle, which are also mentioned in \cite{BV MTP}, are that Khintchine's Theorem implies Jarn\'{\i}k's Theorem and also that Dirichlet's Theorem implies the Jarn\'{\i}k--Besicovitch Theorem. We shall elaborate briefly on these examples here, however for further details and proofs we refer the reader to \cite{BRV aspects, BV MTP}.
 
Let us consider what Khintchine's Theorem (Theorem \ref{Khintchine's Theorem}) can tell us when our approximating function $\psi: \N \to \R^+$ is given by $\psi(q) = q^{-\tau}$ for some $\tau > 1$. In this case we will write $\cA(\tau)$ in place of $\cA_{1}(\psi)$ and we will refer to the points in $\cA(\tau)$ as \itshape $\tau$-approximable points\normalfont. For any $\tau > 1$ it can be seen that the sum of interest in Khintchine's Theorem converges and so all we can infer is that $|\cA(\tau)| = 0$ for all values of $\tau > 1$. However, in this case, we are still able to distinguish the ``sizes'' of these sets thanks to the Jarn\'{\i}k--Besicovitch Theorem. Jarn\'{i}k and Besicovitch both independently proved the following result regarding the Hausdorff dimension of the $\tau$-approximable points. 

\begin{theorem}[Jarn\'ik~\cite{Jarnik29}, Besicovitch~\cite{Besicovitch ref}] \label{JB Theorem}
Let $\tau>1$. Then 
\[
\dim_{H}(\cA(\tau))=\frac{2}{\tau+1}.
\]
\end{theorem}

In fact it turns out that, using the Mass Transference Principle, the Jarn\'ik--Besicovitch Theorem can be extracted from Dirichlet's theorem. Jarn\'ik later proved a much stronger statement, regarding the Hausdorff-measure of more general sets of $\psi$-approximable points, which can be viewed as the Hausdorff measure analogue of Khintchine's Theorem (Theorem \ref{Khintchine's Simultaneous Theorem}). We state below a modern version of Jarn\'{\i}k's Theorem, see \cite[Theorem 11]{BBDV ref} for a greater discussion of the derivation of this statement.

\begin{theorem}[Jarn\'ik~\cite{Jarnik31}]\label{theo:Jarnik}
Let $\psi: \N \to \R^+$ be an approximation function and let $f$ be a dimension function such that $r^{-k} f(r)$ is monotonic. Then
\[
\Haus^{f}(\cA_{k}(\psi))=
\begin{cases}
0&\text{if \;}\sum_{q=1}^{\infty}q^{k}f\left(\frac{\psi(q)}{q}\right)<\infty,\\
&\\
\Haus^f (\I^k)&\text{if \;}\sum_{q=1}^{\infty}q^{k}f\left(\frac{\psi(q)}{q}\right)=\infty\text{ and $\psi$ is monotonic.}
\end{cases}
\]
\end{theorem}

Setting $\psi(q)=q^{-\tau}$ in Jarn\'{\i}k's Theorem we recover the Jarn\'ik--Besicovitch Theorem and additionally gain knowledge of the Hausdorff measure at the critical value $s_{0}=2/(\tau+1)$, \ie $\Haus^{s_{0}}(\cA(\tau))=\infty$.
Although it may at first be surprising, (the original statement of) Jarn\'{\i}k's Theorem follows directly from (the original statement of) Khintchine's Theorem using the Mass Transference Principle. For a proof see, for example, \cite{ BRV aspects, BV MTP}. We remark here that in the original versions of Theorems \ref{Khintchine's Theorem}, \ref{Khintchine's Simultaneous Theorem} and \ref{theo:Jarnik} various stronger monotonicity conditions were required and note that this is of some relevance when using the Mass Transference Principle to deduce Jarn\'{\i}k's Theorem from Khintchine's Theorem. It is possible to deduce Theorem \ref{Khintchine's Simultaneous Theorem} from Theorem \ref{theo:Jarnik} via the Mass Transference Principle but an additional constraint is required on the monotonicity of $\psi$ in this case.

Apart from these important applications in number theory, the Mass Transference Principle can be used to determine Hausdorff dimension and Hausdorff measure statements for many other constructions. 

We end this section by stating the most general variant of the Mass Transference Principle in the original article of Beresnevich and Velani~\cite{BV MTP} and mentioning one of its applications.
Let $(X,d)$ be a locally compact metric space. Let $g$ be a doubling dimension function and suppose that $X$ is $g$-Ahlfors regular, \ie there exist $0 < c_1 \leq1\leq c_2 < \infty$ and $r_0 > 0$ such that
\[ c_1 g(r) \leq \cH^g(B(x,r)) \leq c_2 g(r)\]
for any ball $B = B(x,r)$ with centre $x \in X$ and radius $r \leq r_0$. In this case, given a ball $B=B(x,r)$  and any dimension function $f$ we define $B^{f,g} = B(x, g^{-1}f(r))$. Note that $B^{g,g} = B$.

\begin{theorem}[Beresnevich -- Velani \cite{BV MTP}] \label{general mtp theorem}
Let $(X,d)$ be a locally compact metric space and let $g$ be a doubling dimension function. Let $\{B_i=B(x_{i},r_{i})\}_{i \in \N}$ be a sequence of balls in $X$ with $r_i \to 0$ as $i \to \infty$ and let $f$ be a dimension function such that ${f(x)}/{g(x)}$ is monotonic. Suppose that, for any ball $B$ in $X$,
\[\cH^g(B \cap \limsup_{i \to \infty}{B^{f,g}_i}) = \cH^g(B).\]
Then, for any ball $B$ in $X$, we have
\[\cH^f(B \cap \limsup_{i \to \infty}{B^{g,g}_i}) = \cH^f(B).\]
\end{theorem}

As an example, Theorem \ref{general mtp theorem} is applicable when $X$ is, say, the standard middle-third Cantor set which we denote by $K$ (\ie $K$ is the set of $x \in [0,1]$ which contain only 0s and 2s in their ternary expansion). In fact, in this case, Levesley, Salp, and Velani \cite{LSV ref} have used Theorem \ref{general mtp theorem} as a tool for proving an assertion of Mahler on the existence of very well approximable numbers in the middle-third Cantor set. It is well known that 
\[|K| = 0 \quad \text{and} \quad \dimh K = \frac{\log{2}}{\log{3}}.\]

As a result of Dirichlet's Theorem, we know that for any $x \in \R$ there exist infinitely many pairs $(p,q) \in \Z \times \N$ for which
\[\left|x - \frac{p}{q} \right| < \frac{1}{q^2}.\]

If the exponent in the denominator of the right-hand side of the above can be improved (\ie increased) for some $x \in \R$ then $x$ is said to be \itshape very well approximable\normalfont; that is, a real number $x$ is said to be very well approximable if there exists some $\varepsilon > 0$ such that 
\begin{align}
\left|x - \frac{p}{q}\right| &< \frac{1}{q^{2+\varepsilon}} \label{vwa inequality}
\end{align}
for infinitely many pairs $(p,q) \in \Z \times \N$. We will denote the set of very well approximable numbers by $\cW$. If, further, (\ref{vwa inequality}) is satisfied for every $\varepsilon > 0$ for some $x \in \R$ then $x$ is called a \itshape Liouville number\normalfont, we will denote by $\cL$ the set of all Liouville numbers. 

It is known that
\begin{align*}
&|\cW|=0\text{, } \dimh(\cW)=1, \\
&|\cL|=0, \text{ and } \dimh(\cL)=0.
\end{align*}  

Regarding the intersection of $\cW$ with the middle-third Cantor set, Mahler is attributed with having made the following claim. 

\begin{assertion*}
There exist very well approximable numbers, other than Liouville numbers, in the middle-third Cantor set; \ie
\[(\cW \setminus \cL) \cap K \neq \emptyset.\]
\end{assertion*}

\begin{remark*}
We refer the reader to \cite{LSV ref} for discussion of the precise origin of this claim and also for some discussion regarding why it is natural/necessary to exclude Liouville numbers from Mahler's assertion.
\end{remark*}

Now, let $\cB = \{3^n: n = 0,1,2,\dots\}$ and, given an approximating function $\psi: \R^+ \to \R^+$, consider the set 
\[\cA_{\cB}(\psi):= \left\{x \in [0,1]: \left|x-\frac{p}{q}\right| < \psi(q) \text{ for infinitely many } (p,q) \in \Z \times \cB\right\}.\] 

Levesley, Salp and Velani have used the general Mass Transference Principle (Theorem \ref{general mtp theorem}) as a tool for establishing the following statement regarding Hausdorff measures of the set $\cA_{\cB}(\psi) \cap K$ in \cite{LSV ref}.

\begin{theorem} \label{LSV theorem}
Let $f$ be a dimension function such that $r^{-\frac{\log{2}}{\log{3}}}f(r)$ is monotonic. Then,
\[
\Haus^{f}(\cA_{\cB}(\psi) \cap K)=
\begin{cases}
0&\text{if \;}\sum_{q=1}^{\infty}{f(\psi(3^n)) \times (3^n)^{\frac{\log{2}}{\log{3}}}}<\infty,\\
&\\
\Haus^f (K)&\text{if \;}\sum_{q=1}^{\infty}{f(\psi(3^n)) \times (3^n)^{\frac{\log{2}}{\log{3}}}}=\infty.
\end{cases}
\]
\end{theorem}

As a consequence of Theorem \ref{LSV theorem} the following corollary may be deduced, for details of how we refer the reader to \cite{LSV ref}.

\begin{corollary}[Levesley -- Salp -- Velani \cite{LSV ref}]
We have 
\[\dimh((\cW \setminus \cL) \cap K) \geq \frac{\log{2}}{2\log{3}}.\]
\end{corollary}

The truth of Mahler's assertion follows immediately from this corollary.

\medskip

Finally, we conclude this section by noting that both the original Mass Transference Principle (Theorem \ref{mtp theorem}) and its generalisation given by Theorem \ref{general mtp theorem} concern lim sup sets arising from sequences of balls. In subsequent sections we will explore what happens when this condition is relaxed. More precisely, we will consider linear forms (Section \ref{linear forms section}) and rectangles (Section \ref{rectangles section}) in the deterministic setting and arbitrary Lebesgue measurable sets in the random setting (Section \ref{sect:randomMTP}).

Inevitably, there are various aspects of the Mass Transference Principle that are not covered in this survey.  For example, we have not touched upon the fundamental connections between the Mass Transference Principle set up and the ubiquitous systems framework as developed in \cite{BDV limsup sets} --- in short, the so-called KGB-Lemma \cite[Lemma 5]{BV MTP} is very much at the heart of both. Although a ubiquitous framework was developed in \cite{BDV limsup sets}, we remark that the idea of a ubiquitous system was introduced earlier in \cite{Dodson-Rynne-Vickers ref} and was further developed in \cite{Bernik-Dodson ref}. For an overview of ubiquity and some of its applications also see \cite{BDV limsup sets, Durand ref} and references within. Another omission from this survey is any mention of mass transference principles in the multifractal setting --- see, for example, \cite{FST ref}. In the interest of brevity we have ultimately opted against the inclusion of such topics and chosen here to only focus on the aspects mentioned above.

\section{Extension to systems of linear forms} \label{linear forms section}

In this section we will consider the extension of the Mass Transference Principle to systems of linear forms and mention some of the associated consequences in the theory of Diophantine approximation. 

\subsection{A mass transference principle for systems of linear forms}

Let $k, m \geq 1$ and $l \geq 0$ be integers such that $k = m + l$. Let $\cR = (R_n)_{n \in \N}$ be a family of planes in $\R^k$ of common dimension $l$. For every $n \in \N$ and $\delta \geq 0$, define
\[ \Delta(R_n,\delta) := \{\mathbf{x} \in \R^k : \dist(\mathbf{x},R_n)<\delta\},\]
where $\dist(\vv x,R_n)=\inf\{\|\vv x-\vv y\|:\vv y\in R_n\}$ and $\|\cdot\|$ is a norm on $\R^k$.

Let $\U : \N \to \R$ be a non-negative real-valued function $ n \mapsto \U_n$ on $\N$ such that $\U_n \to 0$ as $n \to \infty$.
Consider the $\limsup$ set
\[ \La(\U) := \{\mathbf{x} \in \R^k : \mathbf{x} \in \Delta(R_n, \U_n) \text{ for infinitely many } n \in \N\}.\]

In 2006, Beresnevich and Velani also established the following extension of the Mass Transference Principle to systems of linear forms \cite{BV Slicing}.

\begin{theorem}[Beresnevich -- Velani \cite{BV Slicing}] \label{BV Slicing Theorem}
Let $\cR$ and $\U$ be as given above. Let $V$ be a linear subspace of $\R^k$ such that $\dim V = m  = \text{\textnormal{codim }} \cR$,
\begin{enumerate}[(i)]
\item{$V \cap R_n \neq \emptyset$ for all $n \in \N$, and}
\item{$\sup_{n \in \N}{\diam(V \cap \Delta(R_n,1))} < \infty$.}
\end{enumerate}
 Let $f$ and $g : r \to g(r) := r^{-l}f(r)$ be dimension functions such that $r^{-k}f(r)$ is monotonic and let $\Om$ be a ball in $\R^k$. Suppose for any ball $B$ in $\Om$ that
\[\cH^k\left(B \cap \La\left(g(\U)^{\frac{1}{m}}\right)\right) = \cH^k(B).\]
Then, for any ball $B$ in $\Om$,
\[\cH^f(B \cap \La(\U)) = \cH^f(B).\]
\end{theorem}

\begin{remark*}
Note that when $l = 0$ in Theorem \ref{BV Slicing Theorem} we recover the Mass Transference Principle (Theorem \ref{mtp theorem}).
\end{remark*}

Conditions $(i)$ and $(ii)$ essentially say that $V$ should intersect every plane and that the angle of intersection between $V$ and each plane should be bounded away from 0. In other words every plane $R_n$ ought not to be parallel to $V$ and should intersect $V$ in precisely one place. These conditions are technical and come about as a consequence of the ``slicing'' technique used by Beresnevich and Velani to prove Theorem \ref{BV Slicing Theorem} in \cite{BV Slicing} (for a simple demonstration of the idea of ``slicing'' see the proofs of Propositions \ref{rectangles slicing result 1} and \ref{rectangles slicing result 2} in Section \ref{rectangles section}). It was conjectured by Beresnevich et al. \cite[Conjecture E]{BBDV ref} that Theorem \ref{BV Slicing Theorem} should also be true without conditions $(i)$ and $(ii)$. Recently, this conjecture has been settled by Allen and Beresnevich in \cite{AB ref} by using a proof closer in strategy to that used by Beresnevich and Velani to prove the original Mass Transference Principle in \cite{BV MTP}, rather than ``slicing".

\begin{theorem}[Allen -- Beresnevich \cite{AB ref}] \label{mtp for linear forms theorem}
Let $\cR$ and $\U$ be as given above. Let $f$ and $g : r \to g(r) := r^{-l}f(r)$ be dimension functions such that $r^{-k}f(r)$ is monotonic and let $\Om$ be a ball in $\R^k$. Suppose for any ball $B$ in $\Om$ that
\[\cH^k\left(B \cap \La\left(g(\U)^{\frac{1}{m}}\right)\right) = \cH^k(B).\]
Then, for any ball $B$ in $\Om$,
\[\cH^f(B \cap \La(\U)) = \cH^f(B).\]
\end{theorem}

Although Theorem \ref{BV Slicing Theorem} itself has some interesting consequences, see \cite{BBDV ref, BV Slicing}, it seems likely that there will be applications of Theorem \ref{mtp for linear forms theorem} which are unachievable using Theorem \ref{BV Slicing Theorem}. In particular, in Section \ref{general mtp for linear forms section} we record some very general statements obtained in \cite{AB ref} which essentially rephrase Theorem \ref{mtp for linear forms theorem} as statements for transferring Lebesgue measure statements to Hausdorff measure statements for sets of $\psi$-approximable (and $\Psi$-approximable) points.

Before that we state some more concrete applications of Theorems \ref{BV Slicing Theorem} and \ref{mtp for linear forms theorem}; namely, we mention Hausdorff measure analogues of the homogeneous and inhomogeneous Khintchine--Groshev Theorems obtained in \cite{AB ref, BBDV ref}. We also use the Hausdorff measure analogue of the inhomogeneous Khintchine--Groshev Theorem to make some remarks on a theorem of Levesley \cite{Levesley ref}.

\subsection{Hausdorff measure Khintchine--Groshev statements}

Throughout this section, let $n \geq 1$ and $m \geq 1$ be integers and denote by $\mathbb{I}^{nm}$ the unit cube $[0,1]^{nm}\subset\R^{nm}$. 

Given a function $\psi : \N \to \R^+$, let $\cA_{n,m}(\psi)$ denote the set of $\x \in \mathbb{I}^{nm}$ such that
\[|\q\x + \p| < \psi(|\q|)\]
for infinitely many $(\p,\q) \in \Z^m \times \Z^n \setminus \{\0\}$. 

Here, $|\cdot|$ denotes the supremum norm and we think of $\x = (x_{ij})$ as an $n \times m$ matrix and of $\p$ and $\q$ as row vectors. So $\q\x$ represents a point in $\R^m$ given by the system
\[q_1x_{1j} + \dots + q_nx_{nj} \quad (1 \leq j \leq m)\]
of $m$ real linear forms in $n$ variables. We say that the points in $\cA_{n,m}(\psi)$ are \emph{$\psi$-approximable}. As with many sets of interest in Diophantine approximation, $\cA_{n,m}(\psi)$ satisfies an elegant zero-one law with respect to Lebesgue measure.

\begin{theorem}[Khintchine--Groshev Theorem \cite{BV KG}] \label{KG Theorem}
Let $\psi : \N \to \R^+$ and $nm > 1$. Then
$$
|\cA_{n,m}(\psi)| =\left\{
\begin{array}{lcl}
 \displaystyle 0 \quad & \mbox{ \text{if }
 $\sum_{q=1}^{\infty}{q^{n-1}\psi(q)^m} < \infty$,}\\[5ex]
 \displaystyle 1 \quad & \mbox{ \text{if }
 $\sum_{q=1}^{\infty}{q^{n-1}\psi(q)^m} = \infty$.}
\end{array}
\right.
$$
\end{theorem}

The earliest versions of this theorem are attributed to Khintchine and Groshev ~\cite{Groshev38, kh}. These were subject to extra assumptions including monotonicity of $\psi$. Due to the famous counter example of Duffin and Schaeffer \cite{Duffin-Schaeffer ref} we know that if we have $m = n = 1$ then monotonicity cannot be removed. However, when we insist that $nm > 1$ monotonicity of $\psi$ is unnecessary. In the case that $n=1$ or $n \geq 3$ this follows, respectively, from results due to Gallagher~\cite{Gallagher ref} and Schmidt~\cite{Schmidt ref}. In the case that $n \geq 3$ this also follows from a result of Sprind\v{z}uk~\cite[Chapter 1, Section 5]{Sprindzuk ref}. For further information we refer the reader to the detailed survey~\cite{BBDV ref}. It was conjectured by Beresnevich et al. in \cite[Conjecture A]{BBDV ref} that monotonicity should also be unnecessary when $n = 2$. This conjecture was finally settled by Beresnevich and Velani in \cite{BV KG} leaving the above modern statement of the Khintchine--Groshev Theorem, which is the best possible.

Regarding the Hausdorff measure theory, combining Theorem \ref{mtp for linear forms theorem} with Theorem \ref{KG Theorem} yields the following.

\begin{theorem}[Hausdorff measure Khintchine--Groshev Theorem \cite{AB ref, BBDV ref}] \label{Khintchine-Groshev Hausdorff analogue theorem}
Let $\psi: \N \to \R^+$ be any approximating function and let $nm > 1$. Let $f$ and $g: r \to g(r) = r^{-m(n-1)}f(r)$ be dimension functions such that $r^{-nm}f(r)$ is monotonic. Then,
$$
\cH^f(\cA_{n,m}(\psi)) =\left\{
\begin{array}{lcl}
 \displaystyle 0 \quad & \mbox{ \text{if }
 $\sum_{q=1}^{\infty}{q^{n+m-1}g\left(\frac{\psi(q)}{q}\right)} < \infty$,}\\[5ex]
 \displaystyle \cH^f(\mathbb{I}^{nm}) \quad & \mbox{ \text{if }
 $\sum_{q=1}^{\infty}{q^{n+m-1}g\left(\frac{\psi(q)}{q}\right)} = \infty$.}
\end{array}
\right.
$$
\end{theorem}

For completeness, we remark here that before the above statement appeared in \cite{BBDV ref}, the Hausdorff measures and dimension of the sets $\cA_{n,m}(\psi)$ had already been studied by a number of people. Indeed, earlier similar results, albeit subject to further constraints, had already been established. In particular, the first Hausdorff measure result in this direction was obtained by Dickinson and Velani \cite{DV1997 ref} and, even before that, the first Hausdorff dimension results had already been established by Bovey and Dodson \cite{Bovey-Dodson ref}.

Returning to Theorem~\ref{Khintchine-Groshev Hausdorff analogue theorem} we note that the statement in~\cite{BBDV ref} additionally required $\psi$ to be monotonic when $n=2$. At that time it was still unproven that the Khintchine--Groshev Theorem was true without monotonicity in the case that $n=2$. However, it was conjectured in \cite{BBDV ref} that, subject to the validity of the Khintchine--Groshev Theorem without assuming monotonicity when $n=2$ (\ie Theorem~\ref{KG Theorem}), it should be possible to use Theorem \ref{BV Slicing Theorem} to remove this final monotonicity condition also from the Hausdorff measure version of the Khintchine--Groshev theorem, giving Theorem \ref{Khintchine-Groshev Hausdorff analogue theorem}. This conjecture has been verified in \cite{AB ref} where, in fact, two proofs of Theorem \ref{Khintchine-Groshev Hausdorff analogue theorem} are given. The first uses a combination of Theorem \ref{BV Slicing Theorem} and ``slicing'', thus verifying the conjecture, and the second uses Theorem \ref{mtp for linear forms theorem}.

In \cite{AB ref}, a Hausdorff measure version of the inhomogeneous analogue of the Khintchine--Groshev theorem is also established. If we are given an approximating function $\psi: \N \to \R^+$ and a fixed $\y \in \I^m$ then we will denote by $\cA_{n,m}^{\y}(\psi)$ the set of $\x \in \mathbb{I}^{nm}$ such that
\[|\q\x + \p - \y| < \psi(|\q|)\]
for infinitely many $(\p,\q) \in \Z^m \times \Z^n \setminus \{\0\}$.

Regarding the Lebesgue measure of sets $\cA_{n,m}^{\y}(\psi)$ of inhomogeneously $\psi$-approximable points, we have the following inhomogeneous analogue of the Khintchine--Groshev Theorem (Theorem \ref{KG Theorem}).

\begin{theorem}[Inhomogeneous Khintchine--Groshev Theorem] \label{IHKG theorem}
Let $m,n \geq 1$ be integers and let $\y \in \I^m$. If $\psi: \N \to \R^+$ is an approximating function which is assumed to be monotonic if $n = 1$ or $n=2$, then
$$
|\cA_{n,m}^{\y}(\psi)| =\left\{
\begin{array}{lcl}
 \displaystyle 0 \quad & \mbox{ \text{if }
 $\sum_{q=1}^{\infty}{q^{n-1}\psi(q)^m} < \infty$,}\\[5ex]
 \displaystyle 1 \quad & \mbox{ \text{if }
 $\sum_{q=1}^{\infty}{q^{n-1}\psi(q)^m} = \infty$.}
\end{array}
\right.
$$
\end{theorem}

When $n \geq 3$ above theorem is a consequence of a result due to Sprind\v{z}uk \cite[Chapter 1, Section 5]{Sprindzuk ref}. In the other cases, where monotonicity of $\psi$ is imposed, the above statement follows from results of Beresnevich, Dickinson and Velani \cite[Section 12]{BDV limsup sets}. For more detailed discussion we refer the reader to, for example, \cite{BBDV ref}.

By combining Theorem \ref{mtp for linear forms theorem} with Theorem \ref{IHKG theorem} the following Hausdorff measure analogue of Theorem~\ref{IHKG theorem} may be obtained.

\begin{theorem}[Allen -- Beresnevich \cite{AB ref}] \label{Inhomogeneous Khintchine-Groshev Hausdorff analogue theorem}
Let $m,n \geq 1$ be integers, let $\y \in \I^{m}$, and let $\psi: \N \to \R^+$ be an approximating function. Let $f$ and $g: r \to g(r)=r^{-m(n-1)}f(r)$ be dimension functions such that $r^{-nm}f(r)$ is monotonic. In the case that $n = 1$ or $n=2$ suppose also that $\psi$ is monotonically decreasing. Then,
$$
\cH^f(\cA_{n,m}^{\y}(\psi)) =\left\{
\begin{array}{lcl}
 \displaystyle 0 \quad & \mbox{ \text{if }
 $\sum_{q=1}^{\infty}{q^{n+m-1}g\left(\frac{\psi(q)}{q}\right)} < \infty$,}\\[5ex]
 \displaystyle \cH^f(\mathbb{I}^{nm}) \quad & \mbox{ \text{if }
 $\sum_{q=1}^{\infty}{q^{n+m-1}g\left(\frac{\psi(q)}{q}\right)} = \infty$.}
\end{array}
\right.
$$
\end{theorem}

\begin{remark} \label{convergence remark}
We note here that in both Theorems \ref{IHKG theorem} and \ref{Inhomogeneous Khintchine-Groshev Hausdorff analogue theorem} the monotonicity condition on $\psi$ when $n = 1$ or $n=2$ is only required for the divergence cases. For both of these theorems the proofs of the convergence parts follow from standard covering arguments for which no monotonicity conditions need to be imposed. 
\end{remark}

In the next section we show how we can use Theorem \ref{Inhomogeneous Khintchine-Groshev Hausdorff analogue theorem} to provide an alternative proof of a general inhomogeneous Jarn\'{\i}k--Besicovitch Theorem proved by Levesley \cite{Levesley ref}. Furthermore, we are able to comment on the necessity of the monotonicity condition imposed in this theorem of Levesley.

\subsection{A Theorem of Levesley}
The Hausdorff dimension of $\cA_{n,m}^{\y}(\psi)$, in the general inhomogeneous setting, was determined by Levesley in~\cite{Levesley ref}. 
Given a function $f: \N \to \R^+$, the \emph{lower order at infinity of $f$}, usually denoted by $\lambda$, is 
\[\lambda(f) = \liminf_{q \to \infty}{\frac{\log(f(q))}{\log(q)}}.\]

\begin{theorem}[Levesley,~\cite{Levesley ref}] \label{Theorem L}
Let $m, n \in \N$ and let $\psi: \N \to \R^+$ be a monotonically decreasing function. Let $\lambda$ be the lower order at infinity of $1/{\psi}$. Then, for any $\y \in \I^m$,
\[\dimh(\cA_{n,m}^{\y}(\psi)) 
       = \left\{
         \begin{array}{ll}
         m(n-1)+\frac{m+n}{\lambda + 1} & \text{when } \;\;\; \lambda > \frac{n}{m}\;
         ,\\[3ex]
         nm & \text{when } \;\;\; \lambda \leq \frac{n}{m}.
         \end{array}\right.\]
\end{theorem}

Levesley proved the above theorem by considering the cases of $n = 1$ and $n \geq 2$ separately. In both cases his argument uses ideas from ubiquitous systems. These are combined with ideas from uniform distribution in the former case and with a more statistical (``mean-variance'') argument in the latter case. 

Using Theorem~\ref{Inhomogeneous Khintchine-Groshev Hausdorff analogue theorem}, we can give an alternative (and shorter) proof of this theorem in which all values of $m$ and $n$ are dealt with simultaneously. To prove this result using Theorem \ref{Inhomogeneous Khintchine-Groshev Hausdorff analogue theorem} we first establish a useful lemma.

\begin{lemma}\label{liminf lemma}
Let $\psi\,:\,\N \to \R^{+}$ be monotonic. 
Then,
\[
\liminf_{q\to\infty}\frac{-\log (\psi(q))}{\log q}=\liminf_{t\to\infty}\frac{-\log\left(\psi(2^{t})\right)}{\log 2^{t}}.
\]
\end{lemma}
\begin{proof}
Assume first that $\psi$ is non-increasing. Note that $(2^{t})_{t=1}^{\infty}$ is a subsequence of $(q)_{n=1}^{\infty}$ and so trivially, 
\[
\liminf_{q\to\infty}\frac{-\log (\psi(q))}{\log q}\leq\liminf_{t\to\infty}\frac{-\log\left(\psi(2^{t})\right)}{\log 2^{t}}.
\]
It remains to prove the reverse inequality. Suppose for now that $\psi(q) \geq 1$ for all $q \in \N$. In this case, since $\psi(q) \to c$ for some $c \geq 1$ by monotone convergence,
\[ \liminf_{q \to \infty}{\frac{-\log(\psi(q))}{\log{q}}} = 0 = \liminf_{t \to \infty}{\frac{-\log(\psi(2^t))}{\log{2^t}}}.\]
Thus, we may assume that $\psi(q) < 1$ for all sufficiently large $q$. Given $q\in\N$, set $t_{q}$ to be the unique integer satisfying $2^{t_{q}}\leq q<2^{t_{q}+1}$.
Then $\psi(2^{t_{q}})\geq \psi(q)$ and $\log(\psi(2^{t_{q}}))\geq \log(\psi(q))$. Since further $q<2^{t_{q}+1}$ and so $\log q<\log 2^{t_{q}+1}$, we obtain
\begin{align*}
\liminf_{q\to\infty}\frac{-\log (\psi(q))}{\log q}&\geq\liminf_{q\to\infty}\frac{-\log\left(\psi(2^{t_{q}})\right)}{\log 2^{t_{q}+1}}\\
&=\liminf_{q\to\infty}\frac{-\log\left(\psi(2^{t_{q}})\right)}{\log 2^{t_{q}}+\log2}\\
&=\liminf_{t\to\infty}\frac{-\log\left(\psi(2^{t})\right)}{\log 2^{t}},
\end{align*}
as required.

For non-decreasing $\psi$ we can similarly set $t_{q}$ to satisfy $2^{t_{q}-1}\leq q<2^{t_{q}}$ and use the bound $\psi(2^{t_{q}})\geq \psi(q)$; details are left to the reader.
\end{proof}

\begin{proof}[Alternative Proof of Theorem \ref{Theorem L} using Theorem \ref{Inhomogeneous Khintchine-Groshev Hausdorff analogue theorem}]
%To begin with, let us suppose that $\lambda < 0$. Then, by the definition of $\liminf$, for any $\varepsilon > 0$ we must have 
%\[\frac{\log{\left(\frac{1}{\psi(q)}\right)}}{\log{q}} \leq \lambda + \varepsilon\]
%for infinitely many values of $q$. Thus, for infinitely many values of $q$ we have $\psi(q) \geq q^{-(\lambda+\varepsilon)}$. Since $\lambda < 0$, and we may take $\varepsilon \leq |\lambda|$, the result then follows trivially. Hence, we may assume $\lambda \geq 0$.
To avoid confusion throughout the proof, for approximating functions $\psi: \N \to \R^+$ we will write $\lambda_{\psi}$ to denote the lower order at infinity of $1/\psi$. However, when there is no ambiguity we will just write $\lambda$ and omit the additional subscript.

We observe that, since $\psi$ is assumed to be monotonically decreasing, we must have $\lambda_{\psi} \geq 0$. To see this, suppose that $\lambda_{\psi} < 0$. Then, by the definition of the lower order at infinity, it follows that for any $\varepsilon > 0$ we must have $\psi(q) \geq q^{-(\lambda_{\psi} + \varepsilon)}$ for infinitely many values of $q$. In particular, this is true for every $0 < \varepsilon < |\lambda_{\psi}|$ and so we conclude that $\psi$ cannot be monotonically decreasing if $\lambda_{\psi} < 0$. 

We will now show that if the result stated in Theorem~\ref{Theorem L} is true for approximating functions with $\lambda = \frac{n}{m}$ then this implies the validity of the result for approximating functions with $0 \leq \lambda < \frac{n}{m}$. We will then establish the result for approximating functions with $\lambda \geq \frac{n}{m}$. 

For the time being, assume that the conclusion in Theorem~\ref{Theorem L} holds for any monotonically decreasing approximating function with $\lambda =\frac{n}{m}$ and let $\psi: \N \to \R^+$ be a monotonically decreasing approximating function such that $\lambda_{\psi} < \frac{n}{m}$. Consider the function $\Psi: \N \to \R^+$ defined by $\Psi(q) = \min\{\psi(q), q^{-\frac{n}{m}}\}$. Note that $\Psi$ is a monotonically decreasing function (since it is the minimum of two monotonically decreasing functions) and that $\Psi(q) \leq \psi(q)$ for all $q \in \N$ . In particular, we have $\dimh(\cA_{n,m}^{\y}(\Psi)) \leq \dimh(\cA_{n,m}^{\y}(\psi))$. 
%Furthermore, observe that $\lambda_{\Psi} =n/m$. To see that this is the case we first 
Next, note that it follows from the fact that $\Psi(q) \leq q^{-\frac{n}{m}}$ for all $q \in \N$ that $\lambda_{\Psi} \geq \frac{n}{m}$. On the other hand, since $\lambda_{\psi} < \frac{n}{m}$ we know that $\psi(q) \geq q^{-\frac{n}{m}}$ for infinitely many values of $q$. In particular, this implies that we must have $\Psi(q) = q^{-\frac{n}{m}}$ infinitely often and, consequently, that $\lambda_{\Psi} \leq \frac{n}{m}$. Hence, $\lambda_{\Psi} =\frac{n}{m}$ and so, by our assumption, we see that
\[\dimh(\cA_{n,m}^{\y}(\psi)) \geq \dimh(\cA_{n,m}^{\y}(\Psi)) = n(m-1)+\frac{n+m}{\lambda_{\Psi}+1} = nm.\]
Combining this with the trivial upper bound we conclude that $\dimh(\cA_{n,m}^{\y}(\psi)) = nm$, as required.

It remains to be shown that $\dimh(\cA_{n,m}^{\y}(\psi)) = n(m-1)+\frac{n+m}{\lambda+1}$ for monotonically decreasing approximating functions $\psi: \N \to \R^+$ with $\lambda_{\psi} = \lambda \geq \frac{n}{m}$. To this end, suppose $\psi$ is such an approximating function.

Let $s_0 = m(n-1)+\frac{m+n}{\lambda + 1}$ and consider $f_{\delta}(r) = r^{s_0+\delta}$ where $-\frac{n+m}{\lambda + 1} < \delta < \frac{n+m}{\lambda + 1}$. We aim to show that  
\[\cH^{s_0+\delta}(\cA_{n,m}^{\y}(\psi)) = 
         \left\{
         \begin{array}{ll}
         0 & \text{if } \;\;\; \delta > 0\;
         ,\\[3ex]
         \cH^{s_0 + \delta}(\I^{nm}) & \text{if } \;\;\; \delta < 0,
         \end{array}\right.\]
from which the result would follow.

Note that $f_{\delta}(r)$ is a dimension function and $r^{-nm}f_{\delta}(r)$ is monotonic. Let $g_{\delta}(r) = r^{-m(n-1)}f_{\delta}(r) = r^{-m(n-1)+s_0+\delta}$. Since $\delta > -\frac{m+n}{\lambda + 1}$, and so $-m(n-1) + s_0 + \delta > 0$, the function $g_{\delta}(r)$ is a dimension function. Thus $f_{\delta}$ and $g_{\delta}$ satisfy the hypotheses of Theorem~\ref{Inhomogeneous Khintchine-Groshev Hausdorff analogue theorem}.

It follows from the definition of the lower order at infinity that, for any $\varepsilon>0$,
\[ \psi(q) \leq q^{-(\lambda - \varepsilon)} \text{ for all large enough } q\text{ and }\]
\begin{align}
\psi(q) \geq q^{-(\lambda + \varepsilon)} \text{ for infinitely many } q \in \N.\label{liminf bound relation 1}
\end{align}

Combining this with Lemma~\ref{liminf lemma}, we have
\begin{align} \label{liminf upper bound relation 2}
\psi(2^t) \leq 2^{-t(\lambda - \varepsilon)}
\end{align}
for large enough $t$
and, for infinitely many $t$,
\begin{align} \label{liminf lower bound relation 2}
\psi(2^t) \geq 2^{-t(\lambda + \varepsilon)}.
\end{align}

By Theorem \ref{Inhomogeneous Khintchine-Groshev Hausdorff analogue theorem} it follows that to determine $\cH^{f_{\delta}}(\cA_{n,m}^{\y}(\psi))$ we are interested in the behaviour of the sum 
\begin{align} \label{sum}
\sum_{q=1}^{\infty}{q^{n+m-1}g_{\delta}\left(\frac{\psi(q)}{q}\right)} = \sum_{q=1}^{\infty}{q^{n+m-1}\left(\frac{\psi(q)}{q}\right)^{-m(n-1)+s_0+\delta}}.
\end{align}

Observe that, by the conditions imposed on $\delta$, $-m(n-1) + s_0 + \delta > 0$ and also that, by (\ref{liminf bound relation 1}), we have $\psi(q) \leq q^{-(\lambda - \varepsilon)}$ for sufficiently large $q$. Thus, (\ref{sum}) will converge if 
\begin{equation}\label{eqn:upperboundsum}
\sum_{q=1}^{\infty}{q^{n+m-1}(q^{-(\lambda-\varepsilon)-1})^{-m(n-1)+s_0+\delta}} = \sum_{q=1}^{\infty}{q^{n+m-1+(\lambda+1-\varepsilon)(m(n-1)-s_0-\delta)}} < \infty.
\end{equation}

This will be the case if 
\[n+m-1+(\lambda+1-\varepsilon)(m(n-1)-s_0-\delta) < -1\]
which is true if and only if
\[\frac{n+m}{\lambda+1-\varepsilon}+m(n-1) < s_0 + \delta.\]
If $\delta > 0$ we can force the above to be true by taking $\varepsilon$ to be sufficiently small.
Thus we conclude that, for $\delta>0$, (\ref{sum}) converges and consequently $\Haus^{s_{0}+\delta}(\cA_{n,m}^{\mathbf{y}}(\psi))=0$.

Next we establish that (\ref{sum}) diverges when $-\frac{n+m}{\lambda+1}<\delta<0$. First we note, since $\psi$ is monotonically decreasing, that
\begin{align}
\sum_{q=1}^{\infty}{q^{n+m-1}\left(\frac{\psi(q)}{q}\right)^{-m(n-1) + s_0 + \delta}} 
                      &= \sum_{t=1}^{\infty}\;{\sum_{2^{t-1} \leq q < 2^t}{q^{n+m-1}\left(\frac{\psi(q)}{q}\right)^{-m(n-1) + s_0 + \delta}}} \nonumber\\
                      &\geq \sum_{t=1}^{\infty}\;{\sum_{2^{t-1} \leq q < 2^t}{(2^{t-1})^{n+m-1}\left(\frac{\psi(2^t)}{2^t}\right)^{-m(n-1) + s_0 + \delta}}} \nonumber\\
                      &= \sum_{t=1}^{\infty}{2^{t-1}(2^{t-1})^{n+m-1}\left(\frac{\psi(2^t)}{2^t}\right)^{-m(n-1) + s_0 + \delta}} \nonumber\\
%                      &= \sum_{t=1}^{\infty}{(2^{t-1})^{n+m}\left(\frac{\psi(2^t)}{2^t}\right)^{-m(n-1) + s_0 + \delta}} \\  
                      &= \frac{1}{2^{m+n}}\sum_{t=1}^{\infty}{2^{t(n+m)}\left(\frac{\psi(2^t)}{2^t}\right)^{-m(n-1) + s_0 + \delta}}. \label{eqn:asumlowerbound} 
\end{align}

We proceed by showing that, when $\delta<0$, we have for infinitely many $t$ that 
\begin{align} \label{divergence condition}
2^{t(m+n)}\left(\frac{\psi(2^t)}{2^t}\right)^{-m(n-1)+s_0+\delta} \geq 1.
\end{align}

For any $\delta<0$ we can choose $\varepsilon>0$ small enough such that
\[\frac{m+n}{\lambda+1+\varepsilon} + m(n-1) \geq s_0 + \delta.\]
Note that such an $\varepsilon$ exists since we are assuming that $\delta$ is negative. Rearranging, this gives
\[m+n-(\lambda+\varepsilon+1)(-m(n-1)+s_0+\delta) \geq 0\]
and then, exponentiating,
\[2^{t(m+n)}\left(\frac{2^{-t(\lambda+\varepsilon)}}{2^t}\right)^{-m(n-1)+s_0+\delta} \geq 1.\]

Now, by (\ref{liminf lower bound relation 2}) we have $\psi(2^t) \geq 2^{-t(\lambda+\varepsilon)}$ infinitely often and so (\ref{divergence condition}) holds, thus proving the divergence of (\ref{eqn:asumlowerbound}) and hence also the divergence of (\ref{sum}).

Hence, we have shown that
\[\cH^{s_0+\delta}(\cA_{n,m}^{\y}(\psi)) = 
         \left\{
         \begin{array}{ll}
         0 & \text{if } \;\;\; \delta > 0\;
         ,\\[3ex]
         \cH^{s_0 + \delta}(\I^{nm}) & \text{if } \;\;\; \delta < 0.
         \end{array}\right.\]
If $s_0 \leq nm$ then $\cH^{s_0 + \delta}(\I^{nm}) = \infty$ whenever $\delta<0$ and so it would follow that $\dimh(\cA_{n,m}^{\y}(\psi)) = s_0$. We conclude the proof by noting that $s_0 \leq nm$ is equivalent to $\lambda\geq \frac{n}{m}$.
\end{proof}

In Theorem \ref{Theorem L} the approximating function $\psi$ is assumed to be monotonic. However, the main tool in our alternative proof of Theorem \ref{Theorem L} is Theorem \ref{Inhomogeneous Khintchine-Groshev Hausdorff analogue theorem} which requires no monotonicity assumptions on $\psi$ for $n \geq 3$. This leads immediately to the natural question of whether this monotonicity assumption is indeed necessary in Theorem~\ref{Theorem L}. 

Let us consider general (not necessarily monotonic) approximating functions $\psi: \N \to \R^+$ with $\lambda$, the lower order at infinity of $1/{\psi}$, satisfying $\lambda >{n}/{m}$. Assuming no monotonicity conditions on $\psi$ and applying similar arguments to those which we have employed here to re-prove Theorem \ref{Theorem L} we obtain the following  bounds on the Hausdorff dimension of $\cA_{n,m}^{\y}(\psi)$. Although, in the interest of brevity, we omit proof.

\begin{proposition}
Let $m \geq 1$ and $n \geq 3$ be integers. If $\psi: \N \to \R^+$ is any function and $\lambda$ is the lower order at infinity of $1/{\psi}$ then, for any $\y \in \I^m$, if $\lambda > n/{m}$ we have 
\[m(n-1) + \frac{m+n-1}{\lambda+1} \leq \dimh(\cA_{n,m}^{\y}(\psi)) \leq m(n-1) + \frac{m+n}{\lambda+1}.\]
\end{proposition}

We see that the upper and lower bounds in the above do not coincide. Interestingly, it turns out that these bounds are the best possible if one does not assume monotonicity of $\psi$ --- as we will now show. To the best of our knowledge the following result has not been considered before.

\begin{theorem}\label{theo:droppingmonotonicity}
Let $m, n \geq 1$ be integers. Let $\alpha > n/m$ be arbitrary and let $s_0$ be such that 
\[m(n-1) + \frac{m+n-1}{\alpha+1} < s_0 < m(n-1) + \frac{m+n}{\alpha+1}.\]
There exists an approximating function $\psi: \N \to \R^+$ such that $\dimh(\cA_{n,m}^{\y}(\psi)) = s_0$ and $\lambda_{\psi}=\alpha$ (where $\lambda_{\psi}$ is the lower order at infinity of $1/{\psi}$).
\end{theorem}

\begin{proof}
Fix $s_0$ satisfying the inequality in the statement of the theorem. Then, let $J := \{a_{k}: k \in \N\}$, where  $a_{k}=\left\lceil k^{-\gamma}\right\rceil$, 
\[
\gamma:=\frac{2}{n+m-1-(\alpha+1)\left(\frac{n+m}{\beta+1}\right)}
\text{ \;\;\;and\;\;\; }\beta := \frac{n+m}{s_0-m(n-1)}-1.
\]
Note that $\gamma\in(-1,0)$.
Define $\psi: \N \to \R^+$ by 
\[\psi(q) 
       = \left\{
         \begin{array}{ll}
         q^{-\alpha} & \text{if } \;\;\; q \in J\;
         ,\\[3ex]
         q^{-\beta} & \text{if } \;\;\; q \notin J.
         \end{array}\right.\]
         
We show that $\psi$ is an approximating function which satisfies the desired properties of the theorem. 
First, note that
\[
m(n-1) + \frac{n+m}{\alpha + 1} > s_0,\]
which implies that 
\[\frac{n+m}{s_0-m(n-1)} - 1 > \alpha.\] 
In turn, this implies that $\beta > \alpha$ and so $\liminf_{q\to\infty} -\log(\psi(q))/\log(q)=\alpha$, giving $\lambda_{\psi}=\alpha$, as required.

Recall that if $\lambda_{\psi} = \alpha$ then for any $\varepsilon>0$
there exists some $N \in \N$  such that $\psi(q) \leq q^{-(\alpha - \varepsilon)}$ for all $q \geq N$, and $\psi(q) \geq q^{-(\alpha + \varepsilon)}$ for infinitely many $q \in \N$.

To establish that the Hausdorff dimension is $s_{0}$ we note that $\dimh(\cA_{n,m}^{\y}(\psi)) \geq \dimh(\cA_{n,m}^{\y}(q \mapsto q^{-\beta}))$ since $\psi(q) \geq q^{-\beta}$ for all $q$. Furthermore, since $q \mapsto q^{-\beta}$ is a monotonic function with $\lambda_{(q \mapsto q^{-\beta})} = \beta$, by Theorem \ref{Theorem L} we have 
\begin{align*}
\dimh(\cA_{n,m}^{\y}(q \mapsto q^{-\beta})) &= m(n-1) + \frac{m+n}{\beta+1}= s_0.
\end{align*}
Therefore, $\dimh(\cA_{n,m}^{\y}(\psi)) \geq s_0$ and it remains to show that $\dimh(\cA_{n,m}^{\y}(\psi)) \leq s_0$. 

As a consequence of Theorem \ref{Inhomogeneous Khintchine-Groshev Hausdorff analogue theorem} (and Remark \ref{convergence remark}), we only need to verify that for all $\delta > 0$ we have 
\[\sum_{q=1}^{\infty}{q^{n+m-1}\left(\frac{\psi(q)}{q}\right)^{-m(n-1)+s_0+\delta}} < \infty\]
since this would imply that $\cH^{s_0 + \delta}(\cA_{n,m}^{\y}(\psi)) = 0$ and $\dimh(\cA_{n,m}^{\y}(\psi)) \leq s_0 + \delta$.

We note that
\begin{align} \label{split sum}
&\sum_{q=1}^{\infty}{q^{n+m-1}\left(\frac{\psi(q)}{q}\right)^{-m(n-1)+s_0+\delta}} \nonumber \\
            =&\sum_{q \in J}{q^{n+m-1}(q^{-\alpha-1})^{-m(n-1)+s_0+\delta}} + \sum_{q \notin J}{q^{n+m-1}(q^{-\beta-1})^{-m(n-1)+s_0+\delta}} \nonumber \\
            =& \sum_{q \in J}{q^{n+m-1-(\alpha + 1)(s_0+\delta-m(n-1))}} + \sum_{q \notin J}{q^{n+m-1-(\beta + 1)(s_0+\delta-m(n-1))}}.
\end{align}

We consider each of the terms on the right-hand side of (\ref{split sum}) separately and show that each of them converges. We first consider the second sum on the right-hand side of (\ref{split sum}). Since $\delta>0$ we have $s_0 - m(n-1) < s_0+ \delta - m(n-1)$ and hence
\[n+m < \left(\frac{n+m}{s_0-m(n-1)}\right)(s_0+ \delta - m(n-1)).\]
Recalling that 
\[\beta = \frac{n+m}{s_0 - m(n-1)}-1\] 
it follows that 
\[n+m-1-(\beta+1)(s_0+ \delta - m(n-1)) < -1\] 
which is sufficient for the second sum on the right-hand side of (\ref{split sum}) to converge.

For the first sum on the right-hand side of (\ref{split sum}) we make the following observations. First of all notice that
\[n+m-1-(\alpha+1)\left(\frac{n+m}{\beta+1}\right) = n+m-1-(\alpha + 1)(s_0 - m(n-1)).\]
Also note that 
\[
\frac{n+m-1}{\alpha+1}+m(n-1) < s_0 \text{ \;\;\; gives \;\;\; }
n+m-1-(\alpha+1)(s_0-m(n-1)) < 0 .                   
\]
Thus, provided that $\delta$ is sufficiently small,
\begin{align}
\sum_{q \in J}{q^{n+m-1-(\alpha+1)(s_0+\delta-m(n-1))}} 
               &= \sum_{k=1}^{\infty}{a_k^{n+m-1-(\alpha+1)(s_0+\delta-m(n-1))}} \nonumber\\
               &= \sum_{k=1}^{\infty}{\left\lceil k^{-\gamma} \right\rceil^{n+m-1-(\alpha+1)(s_0+\delta-m(n-1))}} \nonumber\\
               &\leq \sum_{k=1}^{\infty}{\left( k^{-\gamma} \right)^{n+m-1-(\alpha+1)(s_0+\delta-m(n-1))}} \label{eqn:lastsumbound}
\end{align}
as $n+m-1-(\alpha+1)(s_0+\delta-m(n-1))<0$ and $\gamma<0$. 

Now, for $\delta > 0$, 
\begin{align*}
\frac{2}{\gamma} &= n+m-1-(\alpha+1)(s_0 - m(n-1)) \\
                 &> n+m-1-(\alpha+1)(s_0+\delta-m(n-1)).
\end{align*}
Hence,
\begin{align} \label{split sum convergence}
1 < \frac{n+m-1-(\alpha+1)(s_0 + \delta -m(n-1))}{n+m-1-(\alpha+1)(s_0 - m(n-1))}
\end{align}
and so (\ref{eqn:lastsumbound}) converges since $\left( k^{-\gamma} \right)^{n+m-1-(\alpha+1)(s_0+\delta-m(n-1))}< k^{-2}$. Consequently, since both the component sums converge, it follows that (\ref{split sum}) converges, \ie
\[\sum_{q=1}^{\infty}{q^{n+m-1}\left(\frac{\psi(q)}{q}\right)^{-m(n-1)+s_0+\delta}} < \infty,\]
and we conclude that $\dimh(\cA_{n,m}^{\y}(\psi)) \leq s_0 + \delta$. The desired result follows upon noticing that $\delta>0$ can be taken to be arbitrarily small.
\end{proof}

\subsection{General statements} \label{general mtp for linear forms section}

We conclude this section on the extension of the Mass Transference Principle to systems of linear forms by recording a couple of very general statements established in \cite{AB ref}. Let us consider now the situation where we have approximating functions $\Psi: \Z^n \setminus \{\0\} \to \R^+$ which can depend on $\q$ rather than just $|\q|$. Furthermore, suppose we are also given a fixed inhomogeneous parameter $\y \in \I^m$. We define $\cA_{n,m}^{\y}(\Psi)$ to be the set of $\x \in \I^{nm}$ such that 
\[|\q\x+\p-\y| < \Psi(\q)\]
for infinitely many $(\p,\q) \in \Z^m \times \Z^n \setminus \{\0\}$.

Considering the $\Psi$-approximable points we have the following statement.

\begin{theorem}[Allen -- Beresnevich \cite{AB ref}] \label{general statement: psi depending on q}
Let $\Psi: \Z^n \setminus \{\0\} \to \R^+$ be an approximating function and let $\y \in \I^m$ be fixed.
Let $f$ and $g: r \to g(r)=r^{-m(n-1)}f(r)$ be dimension functions such that $r^{-nm}f(r)$ is monotonic. Let $$
\Theta: \Z^n \setminus \{\0\} \to \R^+\qquad\text{be defined by}\qquad \Theta(\q) := |\vv q|\,g\left(\frac{\Psi(\q)}{|\q|}\right)^{\frac{1}{m}}\,.
$$
Then
$$
|\cA_{n,m}^{\y}(\Theta)| = 1\qquad\text{implies}\qquad \cH^f(\cA_{n,m}^{\y}(\Psi)) = \cH^f(\I^{nm}).
$$
\end{theorem}

Supposing we are interested in the case where we have approximating functions $\psi: \N \to \R^+$ which depend only on $|\q|$ (\ie $\Psi(\q) = \psi(|\q|)$) we can extract the following statement as a corollary to Theorem \ref{general statement: psi depending on q}.

\begin{theorem}[Allen -- Beresnevich \cite{AB ref}] \label{general statement: psi depending on |q|}
Let $\psi: \N \to \R^+$ be an approximating function, let $\y \in \I^m$ be fixed and let $f$ and $g: r \to g(r) = r^{-m(n-1)}f(r)$ be dimension functions such that $r^{-nm}f(r)$ is monotonic. Let
$$
\theta: \N \to \R^+\qquad\text{be defined by}\qquad \theta(r):= r\,g\left(\frac{\psi(r)}{r}\right)^{\frac{1}{m}}\,.
$$
Then
$$
|\cA_{n,m}^{\y}(\theta)| = 1\qquad\text{implies}\qquad \cH^f(\cA_{n,m}^{\y}(\psi)) = \cH^f(\I^{nm}).
$$
\end{theorem}

It is observed in \cite{AB ref} that Theorems \ref{Khintchine-Groshev Hausdorff analogue theorem} and \ref{Inhomogeneous Khintchine-Groshev Hausdorff analogue theorem} follow as corollaries from Theorem \ref{general statement: psi depending on |q|}.
%% ?? too similar to mtp for linear forms paper
In fact, in some sense, Theorems \ref{general statement: psi depending on q} and \ref{general statement: psi depending on |q|} are fairly natural reformulations of Theorem \ref{mtp for linear forms theorem} in terms of, respectively, $\Psi$ and $\psi$-approximable points. 
In essentially the same way that Theorem \ref{mtp for linear forms theorem} may be used to prove Theorem \ref{general statement: psi depending on q} a more general statement can also be obtained.
Namely, suppose we are now given a function $\Psi: \Z^m \times \Z^n \setminus \{\0\} \to \R^+$ which can depend upon both $\p$ and $\q$. Furthermore, suppose we are also given fixed $\Phi \in \I^{mm}$ and $\y \in \I^m$. We denote by $\cM_{n,m}^{\y, \Phi}(\Psi)$ the set of $\x \in \I^{nm}$ for which
\[|\q\x + \p\Phi - \y| < \Psi(\p,\q)\]
holds for $(\p,\q) \in \Z^m \times \Z^n \setminus \{\0\}$ with $|\q|$ arbitrarily large. 

The following statement, which actually includes Theorems \ref{general statement: psi depending on q} and \ref{general statement: psi depending on |q|}, can be made.

\begin{theorem}[Allen -- Beresnevich \cite{AB ref}] \label{general statement: psi depending on (p,q)}
Let $\Psi: \Z^m \times \Z^n \setminus \{\0\} \to \R^+$ be such that
\begin{align} \label{general theorem monotonicity condition}
\lim_{|\q| \to \infty}~\sup_{\p\in\Z^m}{\frac{\Psi(\p,\q)}{|\q|}} = 0\,,
\end{align}
and let $\y \in \I^m$ and $\Phi \in \I^{mm} \setminus \{\0\}$ be fixed. Let $f$ and $g: r \to g(r)=r^{-m(n-1)}f(r)$ be dimension functions such that $r^{-nm}f(r)$ is monotonic. Let
$$
\Theta: \Z^m \times \Z^n \setminus \{\0\} \to \R^+\qquad\text{be defined by}\qquad \Theta(\p,\q) = |\q|\,g\left(\frac{\Psi(\p,\q)}{|\q|}\right)^{\frac{1}{m}}.
$$
Then
$$
|\cM_{n,m}^{\y, \Phi}(\Theta)| = 1\qquad\text{implies}\qquad\cH^f(\cM_{n,m}^{\y,\Phi}(\Psi)) = \cH^f(\I^{nm})\,.
$$
\end{theorem} 

The above theorem not only allows us to consider the usual homogeneous and inhomogeneous settings of Diophantine approximation for systems of linear forms (see \cite{BBDV ref}) but also allows us to consider Hausdorff measure statements where we may have some restrictions on our ``approximating points'' $(\p,\q)$. As an example, recently Dani, Laurent and Nogueira have established Lebesgue measure ``Khintchine--Groshev'' type statements for sets of $\psi$-approximable points where they have imposed certain primitivity conditions on their ``approximating points'' \cite{DLN ref}. In \cite{AB ref}, Theorem \ref{general statement: psi depending on (p,q)} has been used to establish Hausdorff measure versions of these results.

\section{Extension to rectangles} \label{rectangles section}

Another very natural situation, not covered by the setting of systems of linear forms, for which we might hope for some kind of mass transference principle is when our $\limsup$ sets of interest are defined by sequences of rectangles. For example, this is of interest when we consider weighted simultaneous approximation. Recently some progress has been made in this direction by Wang, Wu and Xu~\cite{WWX ref}.

\subsection{A mass transference principle from balls to rectangles}

Throughout this section let $k \in \N$ and, as usual, denote by $\I^k$ the unit cube $[0,1]^k$ in $\R^k$. Given a ball $B = B(\x,r)$ in $\R^k$ of radius $r$ centred at $\x$ and a $k$-dimensional real vector $\veca = (a_1, a_2, \dots, a_k)$ we will denote by $B^{\veca}$ the rectangle with centre $\x$ and side-lengths $(r^{a_1}, r^{a_2}, \dots, r^{a_k})$. Given a sequence $(\x_n)_{n \in \N}$ of points in $\I^k$ and a sequence $(r_n)_{n \in \N}$ of positive real numbers such that $r_n \to 0$ as $n \to \infty$ we define
\[W_{0} = \{\x \in \I^k: \x \in B_n = B(\x_n,r_n) \text{ for infinitely many } n \in \N\}.\] 
For any $\veca \in \R^k$ we will also write
\[W_{\veca} = \{\x \in \I^k: \x \in B^{\veca}_n \text{ for infinitely many } n \in \N\}.\]
 In \cite{WWX ref}, Wang, Wu and Xu established the following mass transference principle.

\begin{theorem}[Wang -- Wu -- Xu \cite{WWX ref}] \label{WWX Theorem 1.2}
Let $(\x_n)_{n \in \N}$ be a sequence of points in $\I^k$ and $(r_n)_{n \in \N}$ be a sequence of positive real numbers such that $r_n \to 0$ as $n \to \infty$. Let $\veca = (a_1, a_2, \dots, a_k)\in\R^{k}$ be such that $1 \leq a_1 \leq a_2 \leq \dots \leq a_k$. Suppose that $|W_0| = 1$. Then,
\[\dimh W_{\veca} \geq \min_{1 \leq j \leq k}\left\{\frac{k + ja_j - \sum_{i=1}^{j}{a_i}}{a_j}\right\}.\]
\end{theorem}

Furthermore, if we have the additional constraint $a_d > 1$, Wang, Wu and Xu are also able to say something about the Hausdorff measure of $W_{\veca}$ at the critical value
\begin{equation}\label{eq:critValue}
s:= \min_{1 \leq j \leq k}\left\{\frac{k + ja_j - \sum_{i=1}^{j}{a_i}}{a_j}\right\}.
\end{equation}

\begin{theorem}[Wang -- Wu -- Xu \cite{WWX ref}] \label{WWX Theorem 1.3}
Assume the same conditions as in Theorem \ref{WWX Theorem 1.2}. If the additional constraint that $a_d > 1$ holds, then
\[\cH^s(W_{\veca}) = \infty.\]
\end{theorem}

Essentially, the results of Wang, Wu and Xu allow us to pass from a full Lebesgue measure statement for a $\limsup$ set defined by a sequence of balls to a Hausdorff measure statement for a $\limsup$ set defined by an associated sequence of rectangles. As an application, Wang, Wu and Xu demonstrate how Theorem \ref{WWX Theorem 1.2} may be applied to obtain the Hausdorff dimension of the following set of weighted simultaneously well-approximable points. Let $\tau = (\tau_1, \tau_2, \dots, \tau_k) \in \R^k$ be such that $\tau_i > 0$ for $1 \leq i \leq k$ and denote by $W_{k}(\tau)$ the set of points $\x = (x_1, x_2, \dots, x_k) \in \I^k$ such that
\begin{align} \label{rectangle inequality}
|qx_i + p_i| < q^{-\tau_i}, \quad 1 \leq i \leq k,
\end{align}
for infinitely many $(\p,q) \in \Z^k \times \N$. The following is derived in \cite{WWX ref} as a corollary to Theorem \ref{WWX Theorem 1.2}. 

\begin{corollary}[Wang -- Wu -- Xu \cite{WWX ref}] \label{WWX corollary}
Let $\tau = (\tau_1, \tau_2, \dots, \tau_k) \in \R^k$ be such that $\frac{1}{k} \leq \tau_1 \leq \tau_2 \leq \dots \leq \tau_k$, then
\[\dimh(W_{k}(\tau)) = \min_{1 \leq j \leq k}\left\{\frac{k + 1 + j\tau_j - \sum_{i=1}^{j}{\tau_i}}{1 + \tau_j}\right\}.\]
\end{corollary}

While the proof of Corollary \ref{WWX corollary} given in \cite{WWX ref} is novel and is a neat application of Theorem \ref{WWX Theorem 1.2} the result itself was already previously known. In fact, Corollary \ref{WWX corollary} is a special case of an earlier more general theorem due to Rynne~\cite{Rynne 98} which we now state. 

Suppose $Q$ is an arbitrary infinite set of natural numbers and, given $\tau \in \R^k$, let $W_{k}^{Q}(\tau)$ denote the set of points $\x \in \I^k$ for which the inequalities in (\ref{rectangle inequality}) hold for infinitely many pairs $(\p,q) \in \Z^k \times Q$, hence $W_{k}^{\N}(\tau) = W_{k}(\tau)$. Define 
\[\nu(Q) = \inf\left\{\nu \in \R: \sum_{q \in Q}{q^{-\nu}} < \infty \right\}\]
and let $\sigma(\tau) = \sum_{i=1}^{k}{\tau_i}$.

\begin{theorem}[Rynne \cite{Rynne 98}] \label{Theorem R}
Let $\tau = (\tau_1, \tau_2, \dots, \tau_k) \in \R^k$ be such that $0 < \tau_1 \leq \tau_2 \leq \dots \leq \tau_k$. Let $Q \subseteq \N$ be arbitrary and suppose that $\sigma(\tau) \geq \nu(Q)$. Then,
\[\dim{W_{k}^{Q}(\tau)} = \min_{1 \leq j \leq k}\left\{\frac{k+\nu(Q)+j\tau_j - \sum_{i=1}^{j}{\tau_i}}{1+\tau_j}\right\}.\]
\end{theorem}

We may easily recover Corollary \ref{WWX corollary} by taking $Q = \N$ in Theorem \ref{Theorem R} and noting that $\nu(\N) = 1$. Since the hypotheses of Corollary \ref{WWX corollary} demand that $\tau_i \geq \frac{1}{k}$ for all $1 \leq i \leq k$ we see that the condition $\sigma(\tau) \geq \nu(Q)$ in Theorem \ref{Theorem R} is also satisfied.

Sets such as $W_{k}(\tau)$ and variations on $W_{k}^{Q}(\tau)$ have been studied in some depth, with particular attention paid to the question of determining their Hausdorff dimension, even before the work of Rynne \cite{Rynne 98}. For example, consider $\tau \in \R$ for some $\tau > 1$. Then the set $W_{1}^{\N}(\tau) = W_{1}(\tau)$ coincides precisely with the set $\cA(\tau)$ considered in the Jarn\'{\i}k--Besicovitch Theorem (Theorem \ref{JB Theorem}). For an overview of some other earlier work in this direction we direct the reader to the discussion given in \cite{Rynne 98} and references therein.

\subsection{Rectangles to rectangles}

The original Mass Transference Principle (Theorem \ref{mtp theorem}) allows us to transition from Lebesgue to Hausdorff measure statements when our original and ``transformed'' $\limsup$ sets are defined by sequences of balls, \ie it allows us to go from ``balls to balls''. Theorem \ref{WWX Theorem 1.2} allows us to go from ``balls to rectangles''. Another goal which we might like to achieve, which is not covered by any of the frameworks mentioned so far, would be to prove a similar mass transference principle where we both start and finish with $\limsup$ sets arising from sequences of rectangles, \ie from ``rectangles to rectangles''. 

\begin{problem}
Does there exist a mass transference principle, similar to Theorem \ref{mtp theorem} or Theorem \ref{WWX Theorem 1.2}, where both the original and transformed $\limsup$ sets are defined by sequences of rectangles?
\end{problem}

Although in the most general settings this problem remains open we survey what can be said in a few special cases. 

In \cite{BV Slicing} Beresnevich and Velani employ a ``slicing'' technique, which uses a combination of a slicing lemma and the original Mass Transference Principle, to prove Theorem \ref{BV Slicing Theorem}. We show how an appropriate combinination of these two results can also be applied to considering the problem of proving a mass transference principle for rectangles. We proceed by stating the ``Slicing Lemma'' as given by Beresnevich and Velani in \cite{BV Slicing}.

\begin{lemma}[Slicing Lemma \cite{BV Slicing}] \label{Slicing Lemma}
Let $l, k \in \N$ be such that $l \leq k$ and let $f$ and $g: r \to r^{-l}f(r)$ be dimension functions. Let $A \subset \R^k$ be a Borel set and let $V$ be a $(k-l)$-dimensional linear subspace of $\R^k$. If for a subset $S$ of $V^{\perp}$ of positive $\cH^l$-measure
\[\cH^g(A \cap (V+b)) = \infty \quad \text{for all } b \in S,\]
then $\cH^f(A) = \infty$.
\end{lemma}

Suppose that $(\x_n)_{n} = (x_{n,1},x_{n,2},\dots,x_{n,k})_{n}$ is a sequence of points in $[0,1]^k$. Let $(r_n^1)_n, (r_n^2)_n, \dots, (r_n^k)_n$ be sequences of positive real numbers and suppose that $r_n^1 \to 0$ as $n \to \infty$. Let
\[H_n = \prod_{i=1}^k{B(x_{n,i}, r_n^i)}\]
be a sequence of rectangles in $[0,1]^k$, where $\prod_{i=1}^{k}{A_i} = A_1 \times A_2 \times \dots \times A_k$ is the Cartesian product of subsets $A_i$ of $\R^k$. Let $\alpha > 1$ be a real number and define another sequence of rectangles by
\[h_n = B(x_{n,1},(r_n^1)^{\alpha}) \times \prod_{i=2}^k{B(x_{n,i}, r_n^i)}\]
so $h_n$ is essentially a ``shrunk'' rectangle corresponding to $H_n$ from the original sequence. Note that in this case we only allow shrinking of the original rectangle in one direction. Then, we are able to establish the following.

\begin{proposition} \label{rectangles slicing result 1}
Let the sequences $H_n$ and $h_n$ be as given above and further suppose that $|\limsup_{n \to \infty}{H_n}| = 1$. Then,
\[\dimh\left(\limsup_{n \to \infty}{h_n}\right) \geq \frac{1}{\alpha} + k - 1.\]
\end{proposition}

\begin{proof}
Let $V = \{\x=(x_1,\dots,x_k) \in [0,1]^k: x_i=0 \quad \text{for all } i \neq 1\}$. Since $|\limsup_{n \to \infty}{H_n}| = 1$, for Lebesgue almost every \[\vecb \in \{\x = (x_1,\dots,x_k) \in [0,1]^k : x_1 = 0\}\] we have
\[ |(V + \vecb) \cap \limsup_{n \to \infty}{H_n}| = 1.\]
Let us fix a $\vecb$ for which this holds and let $W = V + \vecb$. Now, $\limsup_{n \to \infty}{H_n} \cap W$ can be written as the $\limsup$ set of a sequence of balls $B_j = B(x_{n_j,1},r_{n_j}^1)$ with radii $r_{n_j}^1$. Note that $|\limsup_{j \to \infty}{B_j} \cap W| = 1$. For each $j$ also let $b_j = B(x_{n_j,1},(r_{n_j}^1)^{\alpha})$ and note that 
\[\limsup_{j \to \infty}{b_j} \cap W = \limsup_{n \to \infty}{h_n} \cap W.\] 
In accordance with our earlier notation, $b_j^s = B(x_{n_j,1}, (r_{n_j}^1)^{\alpha s})$. Therefore, if $s \leq \frac{1}{\alpha}$ then $(r_{n_j}^1)^{\alpha s} \geq r_{n_j}^1$ for sufficiently large $j$ and so 
\[b_j^s \supseteq B_j \quad \text{and } \quad |\limsup_{j \to \infty}{b_j^s} \cap W| = 1.\] 
Thus, for any $s \leq \frac{1}{\alpha}$ we may use the Mass Transference Principle to conclude that for any ball $B \subseteq W$ we have
\[\cH^s(\limsup_{j \to \infty}{b_j} \cap B) = \cH^s(B).\]
In particular, since $s \leq \frac{1}{\alpha} < 1$, this means
\[\cH^s(\limsup_{n \to \infty}{h_n} \cap W) = \cH^s(W) = \infty.\]
Since this is the case for Lebesgue almost every $\vecb \in \{\x = (x_1,\dots,x_k): x_1 = 0\}$ we can use the slicing lemma to conclude that
\[\cH^{s'}(\limsup_{n \to \infty}{h_n}) = \infty\]
for all $s' \leq \frac{1}{\alpha} + k - 1$. 
Therefore, it follows that 
\[\dimh\left(\limsup_{n \to \infty}{h_n}\right) \geq \frac{1}{\alpha} + k - 1.\] 
\end{proof}

Using Theorem \ref{WWX Theorem 1.3} in place of Theorem \ref{mtp theorem} we are actually able to extend this argument a little further. Again, let $(\x_n)_n = (x_{n,1}, x_{n,2}, \dots, x_{n,k})_n$ be a sequence of points in $[0,1]^k$ and let $(r_n^1)_n, (r_n^2)_n, \dots, (r_n^k)_n$ be sequences of positive real numbers. Suppose that for some $1 \leq k_0 \leq k$ we have $r_n^1 = r_n^2 = \dots = r_n^{k_0}$ for all $n \in \N$ and also that $r_n^1 \to 0$ as $n \to \infty$. Let
\[H_n = \prod_{i=1}^k{B(x_{n,i},r_n^i)}\]
define a sequence of rectangles in $[0,1]^k$. 
Next, let $1 \leq a_1 \leq a_2 \leq \dots \leq a_{k_0}$ be real numbers and suppose $a_{k_{0}}>1$. For each rectangle $H_n$ in our original sequence we define a corresponding ``shrunk'' rectangle
\[h_n = \prod_{i=1}^{k_0}{B(x_{n,i},(r_n^i)^{a_i})} \times \prod_{i=k_0+1}^{k}{B(x_{n,i},r_n^i)}.\]

In this case we are able to prove the following.

\begin{proposition} \label{rectangles slicing result 2}
Let the sequences of rectangles $H_n$ and $h_n$ be as given above and further suppose that $|\limsup_{n \to \infty}{H_n}| = 1$. Then,
\[\dimh\left(\limsup_{n \to \infty}{h_n}\right) \geq \min_{1 \leq j \leq k_0}\left\{\frac{k_0 + ja_j - \sum_{i=1}^{j}{a_i}}{a_j} + k - k_0\right\}.\]
\end{proposition}

\begin{proof}
Let $V = \{\x = (x_1,x_2, \dots, x_k) \in [0,1]^k: x_i = 0 \text{ for all } i \geq k_0 + 1\}$. Since $|\limsup_{n \to \infty}{H_n}| = 1$, for almost every \[\vecb \in \{\x = (x_1, x_2, \dots, x_k) \in [0,1]^k: x_i = 0 \text{ for all } i \leq k_0\}\] we have
\[ |(V + \vecb) \cap \limsup_{n \to \infty}{H_n}| = 1.\]
Let us fix a $\vecb$ for which this holds and let $W = V + \vecb$. As before, $\limsup_{n \to \infty}{H_n} \cap W$ can be written as a sequence of $k_0$-dimensional balls $B_j = B(\x_{n_j}^{k_0},r_{n_j}^1)$ with radii $r_{n_j}^1$($=r_{n_j}^2 = \dots = r_{n_j}^{k_0}$) and centres $\x_{n_j}^{k_0} = (x_{n_j,1}, x_{n_j,2},\dots,x_{n_j,k_0})$. Note that $|\limsup_{j \to \infty}{B_j} \cap W| = 1$.

This time, for each $j$ let 
\[b_j = \prod_{i=1}^{k_0}{B(x_{n_j,i},(r_{n_j}^i)^{a_i})}\] 
and note that 
\[\limsup_{j \to \infty}{b_j} \cap W = \limsup_{n \to \infty}{h_n} \cap W.\] 
By Theorem \ref{WWX Theorem 1.3} it follows that
\[\cH^s(\limsup_{n \to \infty}{h_n} \cap W) = \infty\]
where
\[s := \min_{1 \leq j \leq k_0}\left\{\frac{k_0 + ja_j - \sum_{i=1}^{j}{a_i}}{a_j}\right\}.\]
Since this is the case for almost every \[\vecb \in \{\x = (x_1, x_2, \dots, x_k) \in [0,1]^k: x_i = 0 \text{ for all } i \leq k_0\}\] we may use Lemma \ref{Slicing Lemma} (with $l = k - k_0$) to conclude that 
\[\cH^{s'}(\limsup_{n \to \infty}{h_n}) = \infty\]
where 
\[s' := \min_{1 \leq j \leq k_0}\left\{\frac{k_0 + ja_j - \sum_{i=1}^{j}{a_i}}{a_j} + k - k_0\right\}.\]
Hence
\[\dimh\left(\limsup_{n \to \infty}{h_n}\right) \geq s',\]
as required.
\end{proof}

A disadvantage of using the ``slicing'' arguments above is that we have to impose quite strict conditions on both the original and transformed rectangles. Namely, the sides of the original rectangle which are permitted to ``shrink'' have to be of the same initial length (but can shrink at different rates). Meanwhile, the rest of the sides of the original rectangle are not allowed to ``shrink'' at all when passing to the corresponding transformed rectangle. We conclude this section by considering one more situation where all sides of the original rectangles may have different lengths and are all allowed to ``shrink'' in a specified manner.
Let 
\[H_n = \prod_{i=1}^{k}{B(\x_{n,i}, r_n^{t_i})}\]
be a sequence of rectangles in $[0,1]^k$ with $1 \leq t_i$ for $1 \leq i \leq k$. 

Let the corresponding ``shrunk'' rectangles be defined as
\[ h_n = \prod_{i = 1}^{k}{B(\x_{n,i},r_n^{a_it_i})},\]
where $1 \leq a_i$ for $1 \leq i \leq k$. Suppose without loss of generality that $1 \leq a_1t_1 \leq a_2t_2 \leq \dots \leq a_kt_k$. 

By using the ``natural'' covers of $\limsup_{n \to \infty}{h_n}$ we can get an upper bound for the Hausdorff dimension of this $\limsup$ set; namely, we see that
\begin{equation}\label{eq:upperbound}
\dimh\left(\limsup_{n \to \infty}{h_n}\right) \leq \min_{1 \leq j \leq k}\left\{\frac{\sum_{i=1}^{k}{t_i + ja_jt_j} - \sum_{i=1}^{j}{a_it_i}}{a_jt_j}\right\}. 
\end{equation}

\begin{problem}
Under what conditions do we get a lower bound which coincides with the upper bound given above?
\end{problem} 

\begin{remark*}
Throughout this section we have only considered $\limsup$ sets of rectangles which are all aligned. It would also be natural to consider situations where this is not necessarily the case.
\end{remark*}

\section{Random Mass Transference Principles}\label{sect:randomMTP}
It is a well known phenomenon that introducing randomness to a construction can simplify results by ``smoothing'' out almost impossible values in the probability space that cause problems in deterministic settings.
In this section we will summarise recent progress on random analogues of the statements presented in the preceding sections. We note that the assumptions required are much weaker but with the caveat that randomness has to be introduced somewhere and precise number theoretic results cannot be recovered. The random covering sets that we will mention, as well as the random and deterministic sets we will relate to $\limsup$ sets, have a long history of their own. 
While we highlight their connection to the $\limsup$ sets mentioned in the previous sections and focus on their similarities, we note that the methods used in their proofs differ quite substantially.

We first consider a problem known as the (random) moving target problem. Let $(X,\mu)$ be a probability space, where $X$ is a complete metric space. Let $\{B_{i}\}_{i\in\N}=\{B(x,r_{i})\}_{i\in\N}$ be a sequence of balls centred at $x\in X$ such that $r_{i}\to 0$ as $i\to\infty$. We are interested in the following question.
\begin{problem}
Let $\{\widetilde B_{i}\}_{i\in\N}=\{B(x+a_i,r_{i})\}_{i \in \N}$ be a sequence of balls with random centres $x+a_i$, where $a_{i} \in X$ are chosen independently according to the probability measure $\mu$. Under what conditions can we deduce a measure statement for the $\limsup$ set $E(B_{i})=\limsup_{i\to\infty}\widetilde B_{i}$?
\end{problem}

If $X=\T^{1}$ is the circle and $\mu$ is the uniform measure, one answer to that question should be familiar. It is the Borel--Cantelli Lemma.
\begin{lemma}[Borel--Cantelli Lemma]
Let $X=\T^{1}$ and let $\{B_{i}\}_{i\in\N}=\{B(x,r_{i})\}_{i \in \N}$ be a sequence of balls centred at $x\in X$ such that $r_{i}\to 0$ as $i\to\infty$. Let $(a_i)_{i \in \N}$ be a sequence of random translations chosen according to the uniform measure $\mu$. Then, we again consider $E(B_{i})=\limsup_{i\to\infty}\widetilde B_{i}$ and for almost every choice of sequence $(a_i)_{i \in \N}$ with respect to the product measure $\mu^\N$, we have
\[
\lvert E(B_{i})\rvert=\begin{cases}
0	&	\text{if }\sum_{i=1}^{\infty} r_{i}<\infty,\\
&\\
1	&	\text{if }\sum_{i=0}^{\infty} r_{i}=\infty.
\end{cases}
\]
\end{lemma}
Note that the first implication, \ie that the sum being finite implies zero Lebesgue measure, holds surely for any arbitrary sequence $(a_i)_{i\in \N}$. In particular,  the $a_i$ do not have to be chosen randomly.
Using randomness though, we can make a more precise statement about the Hausdorff dimension when the $\limsup$ set is Lebesgue null.

\begin{theorem}[Fan -- Wu~\cite{Fan04}, Durand~\cite{Durand10}]
Let $X=\T^{1}$, and let $\{B_{i}\}_{i \in \N}$ be a sequence of balls with radii $r_i$ such that $r_{i}\to0$ as $i\to\infty$. Given this sequence of radii, assume that $\lvert E(B_{i})\rvert=0$ for almost every sequence of uniformly chosen translations $(a_{i})_{i\in\N}\subset\T^{1}$. Then, for almost all sequences of random translations, 
\[
\dim_{H}E(B_{i})=\min\{1,\,s_{0}\},
\]
where
\begin{equation}\label{eqn:convergenceofsum}
s_{0}=\inf\left\{s>0 \;:\; \sum_{i=0}^\infty r_{i}^{s}<\infty\right\}.
\end{equation}
\end{theorem}

Related to such statements are results in fractal geometry. We write $\mathcal{I}_{0}=\{T_{1},T_{2},\dots,T_{N}\}$ for a finite collection of contracting similarity maps on $\R^{k}$, \ie 
\[
\lVert T_{i}(x)-T_{i}(y)\rVert=c_{i}\lVert x-y\rVert
\] 
for some $0<c_{i}<1$ for each $1 \leq i \leq N$ for all $x,y\in\R^{k}$, where $\lVert.\rVert$ is the Euclidean norm. The ``best guess'' for the Hausdorff dimension of the unique compact invariant attractor $F\subset\R^{k}$ satisfying $F=\bigcup_{i}T_{i}(F)$ is the \emph{similarity dimension}. The similarity dimension is the unique exponent, $s_0$, satisfying the Hutchinson--Moran formula
\begin{equation}\label{eqn:HutchinsonMoran}
\sum_{i=1}^N c_{i}^{s_0}=1,
\end{equation}
see~\cite{Hutchinson81, Moran46}.
Its relation to (\ref{eqn:convergenceofsum}) can be seen by writing $F$ as a $\limsup$ set
\[
F=\bigcap_{l=1}^{\infty}\;\bigcup_{i=l}^{\infty}\;\bigcup_{j_{1},j_{2},\dots,j_{i}\in\{1,\dots, N\}}T_{j_{1}}\circ T_{j_{2}} \circ \dots \circ T_{j_{i}}(\Delta),
\]
where $\Delta=[-c,c]^{k}$ for some large enough $c\in\R$ such that $F\subseteq\Delta$. 
Notice that $\sum_{j}c_{j}^{s_0+\delta}<1$ for any $\delta>0$ and that $\diam(T_{j_{1}}\circ\dots\circ T_{j_{i}}(\Delta))=c_{j_{1}}\dots c_{j_{i}}c$. So,
\begin{align}
\sum_{i=1}^{\infty}\;\sum_{j_{1},j_{2},\dots,j_{i}\in\{1,\dots, N\}}\diam(T_{j_{1}}&\circ T_{j_{2}} \circ \dots \circ T_{j_{i}}(\Delta))^{s_0+\delta}\nonumber\\
&=c^{s_0+\delta}\sum_{i=1}^{\infty}\;\sum_{j_{1},j_{2},\dots,j_{i}\in\{1,\dots, N\}}c_{j_{1}}^{s_0+\delta}c_{j_{2}}^{s_0+\delta}\dots c_{j_{i}}^{s_0+\delta}\nonumber\\
&=c^{s_0+\delta}\sum_{i=1}^{\infty}\left(\sum_{j}c_{j}^{s_0+\delta}\right)^{i}<\infty\label{eq:selfsimlongeq}
%&=c^{s_0+\delta}\sum_{i=1}^{\infty}\kappa^{i}<\infty,\label{eq:selfsimlongeq}
\end{align}
using additivity. Similarly, if $\delta<0$ the sum above diverges and the similarity dimension $s_0$ in (\ref{eq:selfsimlongeq}) coincides with the expression in (\ref{eqn:convergenceofsum}). We would typically expect the similarity dimension to coincide with the Hausdorff dimension for these sets, but this is not true in general in the deterministic setting and randomisation is one mechanism by which one can get an almost sure equality. We refer the reader to the wide literature on dimension theory of random and deterministic attractors~\cite{Techniques, Falconer ref, Mattila ref}, see also~\cite{Troscheit15} for an overview of self-similar random sets.

Naturally, one is interested in higher dimensional analogues and relaxing the conditions on the covering set $E(B_{i})$. 
Let $X=\T^{k}$  and let $\Delta\subset[0,1]^{k}$ have non-empty interior. Let $T_{i}:\R^{k}\to\T^{k}$ be a linear contraction with singular values $\sigma_{1}(T_{i})\geq \sigma_{2}(T_{i})\geq \dots\geq\sigma_{k}(T_{i})$. Recall that $\sigma_{j}(T_{i})$ is the length of the $j^{\text{th}}$ longest principal semi-axis of the ellipsoid $T_{i}(B(\0,1))$. We define the \emph{singular value function} $\Phi^{t}(T_{i})$ by 
\[
\Phi^{t}(T_{i})=
\begin{cases}
\sigma_{1}(T_{i})\sigma_{2}(T_{i})\dots\sigma_{n}(T_{i})^{t-n+1} & \text{ for }n\leq t+1<n+1 \text{ and }t<k,\\\\
\sigma_{1}(T_{i})\sigma_{2}(T_{i})\dots\sigma_{k}(T_{i})^{t} & \text{ for }t\geq k.
\end{cases}
\]

The Hausdorff dimension of the natural $\limsup$ set appearing in this setting is related to the behaviour of the singular value function.

\begin{theorem}[J\"arvenp\"a\"a -- J\"arvenp\"a\"a -- Koivusalo -- Li -- Suomala~\cite{Jarvenpaa14b}]
Let $(T_{i})_{i\in\N}$ be a sequence of maps as above with $\sigma_{j}(T_{i})\to 0$ as $i\to \infty$ for all $j$. Set 
\[
E(T_{i}):=\limsup_{i}(T_{i}(\Delta)+a_{i}),
\]
where $a_{i}\in\T^{k}$ is a translation chosen independently according to the Lebesgue measure on $\T^{k}$. Then, almost surely,
\begin{equation}\label{eq:affineinfy}
\dim_{H}E(T_{i})=\inf\left\{0<t\leq k \;:\; \sum_{i=1}^{\infty}\Phi^{t}(T_{i})<\infty\right\}.
\end{equation}
\end{theorem}

In particular the sets can now be chosen to be rectangles, as opposed to balls. 
Indeed, even in the deterministic setting considered by Wang, Wu and Xu~\cite{WWX ref} the expression they obtain, namely (\ref{eq:critValue}), coincides with (\ref{eq:affineinfy}). We see this expression appearing yet again in the upper bound (\ref{eq:upperbound}).

The singular value function was first used by Falconer in determining the Hausdorff dimension of self-affine sets~\cite{Falconer88}. Recall that a map is affine if it can be written as $\mathbf{M}x+\mathbf{v}$, for some non-singular matrix $\mathbf{M}\in\R^{k\times k}$ and some vector $\mathbf{v}\in\R^k$. Analogously to the self-similar case, if one considers the unique compact attractor $F$ of a finite collection $\mathcal{I}$ of affine contractions, the ``best guess'' for the Hausdorff dimension is the \emph{affinity dimension} given by the unique value $s\geq 0$ such that
\[
\sum_{T\in\mathcal{I}}\Phi^{s}(T)=1.
\]
In the case where we are given fixed maps and randomly chosen translation vectors the Hausdorff dimension and affinity dimension do coincide, see Falconer~\cite{Falconer88}. More recently, it was shown by B\'ar\'any, K\"aenm\"aki and Koivusalo~\cite{Barany16} that one could alternatively randomise the matrices defining the maps while keeping translation vectors fixed. The problem of determining exact conditions under which self-affine sets have Hausdorff dimension equal to the affinity dimension is still open and much progress has been made towards resolving it; see a recent survey by Falconer~\cite{Falconer survey} and \cite{Falconer15, Kaenmaki16, Morris16} (and references within) for the deterministic setting, and \cite{Fraser16, Gatzouras94, Gui10, Jarvenpaa15, Jordan15, Jordan06, Luzia11, Troscheit15a} for the random setting.

Dropping the linearity of the maps, $T_{i}$, Persson~\cite{Persson15} proved a lower bound for  the Hausdorff dimension of $\limsup$ sets of open sets.
\begin{theorem}[Persson~\cite{Persson15}]
Let $(A_i)_{i \in \N}$ be a sequence of open sets in $\T^k$. Let $V$ be the Riemannian volume on $\T^{k}$ and let
\[
g_{s}(A_{i})=\frac{\lvert A_{i}\rvert^{2}}{\mathcal{E}^{s}(A_{i})}, \quad \text{where }\;\;\;\; \mathcal{E}^{s}(A_{i})=\iint_{A_{i}\times A_{i}}\frac{dV(x)\,dV(y)}{\lvert x-y\rvert^{s}}
\]
is the $s$-energy of $A_{i}$. Then, for the $\limsup$ set $E(A_i)$ we obtain,
\[
\dim_{H}E(A_{i})\geq \inf\left\{0<s\leq k\;:\; \sum_{i=1}^{\infty}g_{s}(A_{i})<\infty\right\}.
\]
\end{theorem}

Now consider the following general set up. Let $U$ and $V$ be open subsets of $\R^k$ and let $T:U\times V\to\R^k$ be a $C^1$ map such that $T(\cdot, y):U\to\R^k$ and $T(x,\cdot):V\to\R^k$ are diffeomorphisms for all $x\in U$ and $y\in V$. Let $D_1 T$ and $D_2 T$ be the derivatives of $T(\cdot, y)$ and $T(x,\cdot)$, respectively. Assume that 
\begin{equation}\label{eq:distortionCond}
\lVert D_i T(x,y)\rVert\leq C_u\;\;\text{ and }\;\;\lVert (D_i T(x,y))^{-1}\rVert\leq C_u
\end{equation}
for some uniform $C_u>0$ and all $i\in\{1,2\}$. 
Let $(A_i)_{i \in \N}$ be a sequence of subsets of $V$ and $(a_i)_{i \in \N}$ be a sequence of points in $U$. The function $T$ defines an interaction between a ``generalised translation'' $a_i$ and a set $A_i$ and embeds them without ``too much distortion'' into $\R^k$. Let $E(T, a_i, A_i)=\limsup_{i\to\infty} T(a_i,A_i)$.
Note that for $T(a_i,y)=x+a_i+y$ this is equivalent to the translates setting considered above.
Feng et al.~\cite{Feng15} proved a (random) mass transference type statement in this general set up.

\begin{theorem}[Feng -- J\"arvenp\"a\"a -- J\"arvenp\"a\"a -- Suomala~\cite{Feng15}]\label{thm:Feng1}
	Let $f$ be a dimension function and for each $i\in\N$ let $a_i \in U$ and let $A_i\subset\Delta\subset V$, where $\Delta$ is compact. Then
		\[
		\sum_{i=1}^\infty \Haus_\infty^f(A_i)<\infty \quad \text{implies} \quad \Haus^f (E(T,a_i,A_i))=0.
		\]
\end{theorem}
Let $\mu$ be a measure on $U$ that is not entirely singular with respect to the Lebesgue measure (see \cite{Mattila ref} for a definition). We denote the natural product measure on all sequences with entries in $U$ by $\Prob=\mu^\N$ and now choose the sequence $(a_{i})_{i\in\N}$ according to $\Prob$. 
Let 
\[
G_f (F)=\sup\{g_f (L) \;:\; L\subset F\text{ and $L$ is Lebesgue measurable with }\lvert L\rvert>0\},
\]
where $g_f$ is the natural extension of $g_s$ to dimension functions $f$,
\[
g_{f}(A_{i})=\frac{\lvert A_{i}\rvert^{2}}{\mathcal{E}^{f}(A_{i})}, \quad \text{where }\;\;\;\; \mathcal{E}^{f}(A_{i})=\iint_{A_{i}\times A_{i}}\frac{dV(x)\,dV(y)}{f(\lvert x-y\rvert)}.
\]
\begin{theorem}[Feng -- J\"arvenp\"a\"a -- J\"arvenp\"a\"a -- Suomala~\cite{Feng15}]
	Suppose the same assumptions as in Theorem~\ref{thm:Feng1}.
	Provided that $\mathcal{E}^f(B(0,R))<\infty$ for all $R>0$ and the $A_i$ are Lebesgue measurable, then
		\[
		\sum_{i=1}^\infty G_f (A_i)=\infty \quad \text{implies} \quad \Haus^f (E(T,a_i,A_i))=\infty\;\;\text{ for }\;\;\Prob-a.e.\ (a_i)_{i \in \N} \in U^\N.
		\]
\end{theorem}
Finally, a set $L$ has positive Lebesgue density if
\[
\liminf_{r\to 0}\frac{\lvert L\cap B(x,r)\rvert}{\lvert B(x,r)\rvert}>0
\]
for all $x\in L$.
\begin{theorem}[Feng -- J\"arvenp\"a\"a -- J\"arvenp\"a\"a -- Suomala~\cite{Feng15}]
	Let $f$ be a dimension function and recall that $V\subset \R^k$.
	Assume that $r^{-k+\epsilon}f(r)$ is decreasing in $r$ for some $\epsilon>0$. Let $h$ be a dimension function such that $h(r)\leq f(r)^{1+\delta}$ for some $\delta>0$ and all $r>0$. 	Under the same assumptions as in Theorem~\ref{thm:Feng1} and provided that the $A_i$ are Lebesgue measurable with positive Lebesgue density we obtain
	\[
	\sum_{i=1}^\infty G_f(A_i) <\infty   \quad \text{implies} \quad   \sum_{i=1}^\infty  \Haus_\infty ^h (A_i)<\infty.
	\]
\end{theorem}
As one can readily see, these latter results hold for $\limsup$ sets of very general subsets.
However, we still require positive Lebesgue density and a ``nice'' measure that is not singular with respect to the Lebesgue measure.
Recently Ekstr\"om and Persson~\cite{Ekstrom16} have made advances in relaxing these conditions on the measures by considering random $\limsup$ sets with random centres chosen according to an arbitrary Borel measure and formulating their results in terms of multifractal formalism. \\
\clearpage
\noindent {\bf Acknowledgements.}
A major portion of this manuscript was prepared when ST visited DA at the University of  York in December 2016. ST thanks York for their hospitality during his stay. DA would like to thank Victor Beresnevich for a number of interesting discussions prompting some of the original results presented in this article. Both authors are grateful to Victor Beresnevich, Henna Koivusalo and Sanju Velani for their helpful comments on earlier drafts of this article. 
At the time of writing DA and ST were supported, respectively, by EPSRC Doctoral Training Grants EP/M506680/1 and EP/K503162/1.

%\clearpage

\end{document}